\documentclass[a4paper, reqno]{amsart}
\usepackage{amsmath}
\usepackage{mathtools}
\usepackage{amssymb}
\usepackage{enumerate}
\usepackage{amsthm}
\usepackage{mathrsfs}
\usepackage{cancel}
\usepackage{array}
\usepackage{color}
\usepackage{geometry}

\DeclareMathAlphabet{\mathpzc}{OT1}{pzc}{m}{it}
\makeatletter
\newcommand*{\rom}[1]{\expandafter\@slowromancap\romannumeral #1@}
\makeatother
\makeatletter

\newcommand{\R}{\mathbb{R}}

\theoremstyle{definition}

\theoremstyle{plain}
\newtheorem*{thm*}{Theorem}
\theoremstyle{plain}
\newtheorem{thm}{Theorem}%[section]
\theoremstyle{plain}
\newtheorem{lem}[thm]{Lemma}
\theoremstyle{plain}

\theoremstyle{plain}
\newtheorem{prop}[thm]{Proposition}
\theoremstyle{remark}

\DeclareMathOperator*{\esssup}{ess\,sup}

\title[]{Nonrelativistic limit of solitary waves for nonlinear Maxwell-Klein-Gordon equations}

\author[S. Jin]{Sangdon Jin}
\address{Department of Mathematical Sciences, KAIST, Daejeon 34141, Republic of Korea}
\email{sdjin@kaist.ac.kr}
\author[J. Seok]{Jinmyoung Seok}
\address{Department of Mathematics, Kyonggi University, Suwon 16227, Republic of Korea}
\email{jmseok@kgu.ac.kr}

\keywords{Maxwell-Klein-Gordon; Schrodinger-Poisson; nonrelativistic limit; solitary wave}

\begin{document}

\begin{abstract}
We study the nonrelativistic limit of solitary waves from Nonlinear Maxwell-Klein-Gordon equations (NMKG) to Nonlinear Schr\"odinger-Poisson equations (NSP). 
It is known that the existence or multiplicity of positive solutions depends on the choices of parameters the equations contain. 
In this paper, we prove that for a given positive solitary wave of NSP, which is found in Ruiz's work \cite{R}, there corresponds a family of positive solitary waves of NMKG under the nonrelativistic limit.  Notably, our results contain a new result of  existence of positive solutions to (NMKG) with lower order nonlinearity.
%a local nonlinear term $u^{p-1}$ determines the existence and nonexistence of solutions for NMKG and NSP, and the existence of solutions for NMKG is unknown for $2<p<3$. In this paper, we prove the convergence of non-relativistic limit of solutions between NMKG and NSP for $2<p<6$, and the existence of radial positive solutions for NMKG for $2<p<3$, which blows up in $H^1(\R^3)$ norm.
\end{abstract}

\maketitle

\section{Introduction}
Nonlinear Maxwell-Klein-Gordon equations are written by 
\begin{equation}\tag{NMKG}
\left\{\begin{aligned}
& D_\alpha D^{\alpha}\phi = (mc)^2\phi -|\phi|^{p-2}\phi, \\
& \partial^{\beta}F_{\alpha\beta} = \frac{q}{c}\text{Im}(\phi\overline{D_\alpha\phi}),
\end{aligned}\right. \quad \text{in } \R^{1+3}.
\end{equation}
where $D_\alpha \coloneqq \partial_\alpha +\frac{q}{c}iA_\alpha, \alpha = 0, 1, 2, 3$ and
$F_{\alpha\beta} \coloneqq \partial_\alpha A_\beta-\partial_{\beta}A_{\alpha}$.
Here, $m>0$ represents the mass of a particle, $q>0$ is a unit charge and $c > 0$ is the speed of light. 
We write $\partial_0 = \frac{\partial}{c\partial t}$, $\partial_i = \frac{\partial}{\partial x_{j}}, j = 1, 2, 3$.
Indices are raised under the Minkowski metric $g_{\alpha\beta} = \text{diag}(-1, 1, 1, 1)$, i.e., $X^{\alpha} \coloneqq g_{\alpha\beta}X_{\beta}$.
If we pay attention to the electrostatic situation, that is, $A_1 = A_2 = A_3 = 0$,
then NMKG is reduced to
\begin{equation}\label{NMKG-electrostatic}
\left\{\begin{aligned}
&\left(-\frac{\partial^2}{c^2\partial t^2}+\Delta\right)\phi -\frac{2q}{c^2}iA_0 \frac{\partial\phi}{\partial t} -\frac{q}{c^2}i\frac{\partial A_0}{\partial t} \phi 
+\left(\frac{q}{c}\right)^2A_0^2\phi = (mc)^2\phi -|\phi|^{p-2}\phi, \\
& -\Delta A_0= \frac{q}{c^2}\textup{Im}\left(\phi\overline{\frac{\partial\phi}{\partial t}}\right) -\left(\frac{q}{c}\right)^2A_0|\phi|^2,
\end{aligned}\right. \quad \text{in } \R^{1+3}.
\end{equation}

This paper is concerned with the nonrelativistic limit for NMKG in electrostatic case.  
By modulating the solution as $\phi(t,x) = e^{imc^2t}\psi(t,x)$, the system of equations \eqref{NMKG-electrostatic} transforms into
\begin{equation}\label{NMKG-electrostatic-modulated}
\left\{\begin{aligned}
&-\frac{\partial^2\psi}{c^2\partial t^2} -2mi\frac{\partial\psi}{\partial t} +\Delta \psi  +2qmA_0\psi -\frac{2q}{c^2}iA_0\frac{\partial\psi}{\partial t}   
-\frac{q}{c^2}i\frac{\partial A_0}{\partial t} \psi +\left(\frac{q}{c}\right)^2A_0^2\psi = -|\psi|^{p-2}\psi, \\
&-\Delta A_0 +\left(\frac{q}{c}\right)^2|\psi|^2A_0 = \frac{q}{c^2}\textup{Im}\left(\psi \overline{\frac{\partial\psi}{\partial t}}\right)-qm|\psi|^2.
\end{aligned}\right.
\end{equation}
Then, taking so-called nonrelativistic limit $c\to\infty$,  the relativistic system \eqref{NMKG-electrostatic-modulated} formally converges to 
 nonlinear equations of Schr\"odinger type, called the nonlinear Schr\"odinger-Poisson equations 
\begin{equation}\tag{NSP}
\left\{\begin{aligned}
&-2mi\frac{\partial\psi}{\partial t} +\Delta \psi  +2qmA_0\psi  = -|\psi|^{p-2}\psi, \\
&-\Delta A_0 = -qm|\psi|^2,
\end{aligned}\right.  \quad \text{in } \R^{1+3}.
\end{equation}
When the nonlinear potential term $|\psi|^{p-2}\psi$ is absent, the rigorous justifications of this limit are carried out by Masmoudi-Nakanish \cite{MN} and Bechouche-Mauser-Selberg \cite{BMS}.
As for the stuides on the nonlinear Klein-Gordon equations without the Maxwell gauge terms ($A_\mu = 0, \mu=0,1,2,3$), we refer to a series for works \cite{MNO, MN1, N}.  
%Refer to some results on nonrelativistic limit for initial value problem.... 

The main interest of this paper lies in investigating  the correspondence between solitary waves of NMKG and NSP under the nonrelativistic limit $c\to\infty$.
During recent two decades, existence theories for solitary waves of NMKG and NSP have been well developed. 
Inserting the standing wave ansatz $\psi(t,x) = e^{-i\mu t}u(x),\, u \in \R$ into \eqref{NMKG-electrostatic-modulated}, we get
\begin{equation}\label{1}
\left\{\begin{aligned}
&-\Delta u+\Big(m^2c^2-\big(\frac{mc^2-\mu}{c}+\frac{q\Phi}{c}\big)^2\Big)u-|u|^{p-2}u=0,  \\
&-\Delta \Phi+\frac{q^2}{c^2}u^2\Phi=-\frac{q}{c}\left(\frac{mc^2-\mu}{c}\right)u^2,
\end{aligned}\right. \qquad \mbox{ in }\R^3.
\end{equation}
Lax-Milgram theorem implies that for each $u\in H^1(\R^3)$,  there exists a unique solution $\Phi_u\in D^{1,2}(\R^3)$ of
\begin{equation}\label{sta3}
-\Delta \Phi+\frac{q^2}{c^2}u^2\Phi=-q(m-\frac{\mu}{c^2})   u^2 \mbox{ in } \R^3.
\end{equation}
Then, by \cite[Proposition 3.5]{BF}, $(u,\Phi)\in H^1(\R^3)\times D^{1,2}(\R^3)$ is a solution of \eqref{1} if and only if $u\in H^1(\R^3)$ is a critical point of $I_c$, and $\Phi=\Phi_u$, where
\begin{align*}
I_c(u)&=\frac12\int_{\R^3}|\nabla u|^2+\Big(2m\mu-\frac{\mu^2}{c^2}\Big)u^2-q\Big(m-\frac{\mu}{c^2}\Big) u^2 \Phi_u dx-\frac{1}{p}\int_{\R^3}|u|^pdx,
\end{align*}
which is a $C^1$ functional on $H^1(\R^3)$.
We note that the system of equations \eqref{1} is equivalent to the single nonlocal equation
\begin{equation}\label{sta1}
-\Delta u+\Big(m^2c^2-\big(\frac{mc^2-\mu}{c}+\frac{q\Phi_u}{c}\big)^2\Big)u-|u|^{p-2}u=0 \mbox{ in }\R^3.
\end{equation}

Before stating the existence results for \eqref{1}, we simplify the parameters by denoting $\bar{m} = mc$, $e = q/c$ and $\omega = (mc^2-\mu)/c$ to rewrite  \eqref{1} as
\begin{equation}\label{2}
\left\{\begin{aligned}
&-\Delta u+\left(\bar{m}^2-\left(\omega+e\Phi\right)^2\right)u-|u|^{p-2}u=0,\\
&-\Delta \Phi+e^2u^2\Phi=-e\omega u^2,
\end{aligned}\right. \qquad \mbox{ in }\R^3,
\end{equation}
where $e > 0$, $\bar{m} > 0$ and $\omega > 0$ such that $\bar{m} > \omega$.
The corresponding action functional is given by
\[
I_{\bar{m},e,\omega}(u)=\frac12\int_{\R^3}|\nabla u|^2+(\bar{m}^2-\omega^2)u^2-e\omega u^2 \Phi_u dx-\frac{1}{p}\int_{\R^3}|u|^pdx.
\]

For fixed $e > 0$, Benci and Fortunato \cite{BF} first  proved by applying critical point theory to $I_{\bar{m},e,\omega}$ that there exist infinitely many solutions of \eqref{2} for $4 < p <6$ and $0 < \omega < \bar{m}$. 
This result is extended by D'Aprile and Mugnai \cite{DM1} to the cases $4\leq p < 6$ and $0 < \omega < \bar{m}$ or $2 < p < 4$ and $0 < \sqrt{2}\omega < \bar{m}\sqrt{p-2}$.
They also proved in \cite{DM} that there exist no nontrivial solutions if $p \leq 2$ or $p \geq 6$ and $0 < \omega \leq \bar{m}$.
In \cite{APP}, Azzollini, Pisani and Pomponio widened the existence range of $\bar{m}, \omega$ for the case $2 < p < 4$ by showing that \eqref{2} admits a nontrivial solution when $0 < \omega < \bar{m}g(p)$, where
\[
g(p) \coloneqq \left\{\begin{array}{rl}
\sqrt{(p-2)(4-p)} & \text{if } 2 < p < 3, \\
1 & \text{if } 3 \leq p < 4.
\end{array}\right.
\]  
Azzollini and Pomponio also focused on the existence of a ground state solution of \eqref{2}.
A critical point of $I_{\bar{m},e,\omega}$ is said to be a ground state solution to \eqref{2} if it minimizes the value of $I_{\bar{m},e,\omega}$ among all nontrivial critical points of $I_{\bar{m},e,\omega}$.
In \cite{AP1}, they showed \eqref{2} admits a ground state solution if $4\le p<6$ and $ 0< \omega < \bar{m}$ or $2 < p < 4$ and $\bar{m}\sqrt{p-1} > \omega\sqrt{5-p}$.
Wang \cite{Wang} established the same result to the range of parameters that $2 < p < 4$ and $0  < \sqrt{h(p)}\omega < \bar{m}$, where 
\[
h(p) \coloneqq 1+\frac{(4-p)^2}{4(p-2)}.
\]  

We now turn to the standing wave solutions for NSP.  We again insert the same ansatz $\psi(t,x) = e^{-i\mu t}u(x),\, u \in \R$ into NSP to obtain
\begin{equation}\label{131}
\begin{aligned}
-\Delta u+&2m\mu u-2qmu \phi-|u|^{p-2}u=0 \mbox{ in }\R^3,\\
&-\Delta \phi=-qm u^2 \mbox{ in } \R^3.
\end{aligned}
\end{equation}
For any $u\in H^1(\R^3)$, there exists a unique $\phi_u\in D^{1,2}(\R^3)$  satisfying
\begin{equation}\label{sta2}
-\Delta \phi_u=-qm u^2 \mbox{ in } \R^3,
\end{equation}  
by Lax-Milgram theorem 
(note that actually $\phi_u=-\frac{qm}{4\pi|x|}\ast u^2$). 
We define the corresponding action integral as
\begin{equation}\label{inftz}
\begin{aligned}
I_\infty(u)&=\frac12\int_{\R^3}|\nabla u|^2+2m\mu u^2-qmu^2\phi_u dx-\frac{1}{p}\int_{\R^3}|u|^pdx.
\end{aligned}
\end{equation}
Then, by \cite[Lemma 3.2]{DM1}, $(u,\phi)\in H^1(\R^3)\times D^{1,2}(\R^3)$ is a solution of \eqref{131} if and only if $u\in H^1(\R^3)$ is a critical point of $I_\infty$, and $\phi=\phi_u$. It is also standard to show that $I_\infty\in C^1(H^1(\R^3),\R)$ and a critical point $u$ of $I_\infty$ satisfies
\begin{equation}\label{a1}
\begin{aligned}
-\Delta u+&2m\mu u-2qmu \phi_u-|u|^{p-2}u=0 \mbox{ in } \R^3.
\end{aligned}
\end{equation}

We summarize some existence results for problem \eqref{a1}.
 D'Aprile-Mugnai \cite{DM1} and Coclite \cite{C} proved the existence of a radial positive solution of \eqref{a1} for $4\le p<6$. On the other hand, using a Pohozaev equality, D'Aprile-Mugnai \cite{DM} showed that there exists no non-trivial solutions of \eqref{a1} for $p\le 2$ or $p\ge 6$. 
By a new approach, Ruiz \cite{R} fills a gap for the range $2<p<4$. More precisely, he proved the following results:
\begin{enumerate}[(i)]
\item ($3<p<6$ and $q>0$) $\exists$ a nontrivial solution, which is a ground state in radial class;
\item ($2<p<3$ and small $q>0$) $\exists$ a nontrivial solution, which is a minimizer of $I_\infty$; 
\item ($2<p\le 3$ and small $q>0$) $\exists$ a nontrivial solution emanating from a ground state solution of
\begin{equation}\label{gr1}
-\Delta u+2m\mu u-|u|^{p-2}u=0 \mbox{ in }\R^3;
\end{equation}
\item ($2 < p \leq 3$ and large $q > 0$) $\not\exists$ nontrivial solution of \eqref{a1}. 
\end{enumerate}
In \cite{AP}, Azzollini and Pomponio constructed a ground state solution  of  \eqref{a1} for $3<p<6$, which is possibly non-radial.
It was shown by Colin and Watanabe \cite{CW} that a ground state is unique and radial up to a translation for small $q > 0$.
This result implies that the solution found by Ruiz coincides with the ground state constructed by Azzollini and Pomponio for small $q > 0$ if $3 < p < 6$.
As far as we know, it is unknown whether the ground states is radial when $q > 0$ is arbitrary.

%\subsection{Statement of main theorems}
Concerning the nonrelativistic limit between solitary waves, one can naturally ask is the following: \\ \\
{\bf Question:} For any positive solution $u$ of \eqref{a1}, is there a corresponding family of positive solutions $u_c$ of \eqref{sta1}, which converges to $u$ as $c\to\infty$?   
\\ \\
  In this paper, we not only give a complete answer to this question,  but also construct blow up solutions to NMKG for $2<p<3$.
Our first theorem states the convergence of nonrelativistic limit of ground states between \eqref{sta1} and \eqref{a1} for $3 < p < 6$.
The theorem contains the existence of a ground state to \eqref{sta1} for $3 <  p < 4$ with arbitrary parameters $m, q, \mu, c > 0$ and $c > \sqrt{\mu/m}$, which is not covered by the aforementioned results of Azzollini-Pomponio \cite{AP1} or Wang \cite{Wang} (see Proposition \ref{gmmq1}).
\begin{thm}[Existence and nonrelativistic limit of ground states]\label{mthm1}
Fix arbitrary $\mu, m, q>0$  and  $3 < p < 6$.  
Then there holds the following:
\begin{enumerate}[$(i)$]
\item There exists a ground state solution of \eqref{sta1} for any $c > \sqrt{\mu/m}$.   
\item Any ground state solution of \eqref{sta1} belongs to $H^2(\R^3)$ and converges to a ground state solution of \eqref{a1} in $H^2(\R^3) $ as $c\rightarrow \infty$, up to a translation and a subsequence. 
\end{enumerate}
\end{thm} 

Based on the strategies proposed in \cite{CHS, CS},   we shall prove the convergence of nonrelativistic limit in Theorem \ref{mthm1} by establishing the following steps:
\begin{enumerate}[1.]
\item Uniform upper estimate of ground energy levels for \eqref{sta1} by the ground energy level for \eqref{a1}, i.e., 
\begin{equation}\label{upper-est}
\limsup_{c\to\infty}E_c \leq E_\infty, 
\end{equation}
where
\[
E_c = \inf\{ I_c(u) ~|~ u \neq 0,\, I_c'(u) = 0 \} \quad \text{and} \quad E_\infty = \inf\{ I_\infty(u) ~|~ u \neq 0,\, I_\infty'(u) = 0 \};
\]
\item Uniform $H^1$ bounds for ground states \{$u_c\}$ of \eqref{sta1} and solvability of its weak limit $u_\infty$ to \eqref{a1};
\item Energy estimates for establishing $u_\infty$ to be a ground state;
\item $H^1$ convergence of $u_c$ to $u_\infty$ and its upgrade to $H^2$.
\end{enumerate} 

A new difficulty arises when we prove the step $1$ in the case $3 < p < 4$. 
It is worth to point out  that we  couldn't  construct a ground state of \eqref{sta1} by using a constrained minimization method for $3 < p < 4$.
It seems not possible to find a suitable constraint working for every admissible parameters $\mu, m, q, c$.
As a consequence, we couldn't compare ground states energy levels between \eqref{sta1} and \eqref{a1}.
To bypass the obstacle, we directly construct a ground state that satisfies the upper estimate \eqref{upper-est}.
That is, we  first  show the existence of a family of nontrivial solutions to \eqref{sta1} satisfying the upper estimate \eqref{upper-est} by applying a deformation argument developed in \cite{BJ}.  Then, by the compactness of a sequence of solutions to \eqref{sta1}, we prove that  aforementioned nontrivial solutions  to \eqref{sta1} is   ground state solutions to \eqref{sta1} (see Proposition \ref{gmmq1}).

%\begin{thm}[Symmetry and uniqueness of ground states]\label{mthm2}
%Assume that $3 < p < 6$. 
%There exists $q_0 > 0$ such that for any $q \in (0, q_0)$ and any sufficiently large $c > 0$, a ground state solution of \eqref{sta1} is radially symmetric and unique, up to a translation.
%\end{thm}

%\begin{thm}\label{mthm1}
%Assume that $4\le p<6$ and $q>0$.  Any ground state solution $ u_{c,q} $ of \eqref{sta1} converges to a ground state solution $u_{\infty,q}$ of \eqref{a1} in $H^1(\R^3) $ as $c\rightarrow \infty$, up to a translation. Moreover, for small $q>0$, a ground state solution $ u_{c,q} $ of \eqref{sta1}  is  radially symmetric and unique, up to a translation.
%\end{thm}

%\begin{thm}\label{mthm2}
%Assume that $3<p<4$ and $q>0$. There exists $c_0>0$ such that for $c\ge c_0$, there exists a ground state solution $ u_{c,q} $ of \eqref{sta1} such that $ u_{c,q}$ converges to a ground state solution $ u_{\infty,q} $ of \eqref{a1} in $H^1(\R^3) $ as $c\rightarrow \infty$, up to a translation. Moreover, for small $q>0$, a ground state solution $u_{c,q}$ is radially symmetric and unique,  up to a translation.
%\end{thm}

The next theorem covers the case that $2 < p < 3$ and $q$ is small.  
We recall the aforementioned results by  Ruiz \cite{R}, which say the existence of two positive radial solutions $u_\infty$ and $v_\infty$ of \eqref{a1}; $u_\infty$ is a perturbation of the ground state to \eqref{gr1} and $v_\infty$ is a global minimizer of $I_\infty$. 
In Theorem \ref{mthm3}, we show the existence of two radial positive solutions $u_{c}$ and $v_{c}$ to \eqref{sta1} such that $u_c$ and $v_c$ converges to $u_\infty$ and $v_\infty$, respectively.

\begin{thm}[Correspondence of two positive solutions for $2 < p < 3$] \label{mthm3}
Assume $2 < p < 3$. Fix arbitrary but sufficiently small $q > 0$ that guarantees the existence of two positive radial solutions $u_\infty$ and $v_\infty$ to \eqref{a1} mentioned above. 
If $c > 0$ is sufficiently large, then there exist two distinct radially symmetric positive solutions $u_c$  and  $v_c$ of \eqref{sta1} such that
\[
\textup{(i) } \lim_{c\to\infty} \|u_c-u_\infty\|_{H^1(\R^3)} = 0, \qquad \textup{ (ii) } \lim_{c\to\infty}\|v_c-v_\infty\|_{H^1(\R^3)}= 0. 
\]
%\[
%0 < \lim_{q\to0}\lim_{c\to\infty}\|u_{c,q}\|_{H^1} < \infty \qquad \textup{and} \qquad \lim_{q\to0}\lim_{c\to\infty}%\|v_{c,q}\|_{H^1} = \infty.
%\]
\end{thm}
%The equation \eqref{sta1} has a radial positive solution $u_{c,q}$, which converges to a ground state solution of \eqref{gr1} for any $2<p<6$ if  $q$ is small.
 In \cite{R}, Ruiz proved that a global minimizer $v_\infty$ of $I_\infty$ blows up in $H^1$ as $q\rightarrow 0$, which implies that the solution $v_c$ constructed in Theorem \ref{mthm3} blows up in $H^1$ as $q\rightarrow 0$ and $c\rightarrow \infty$.
We point out that Theorem \ref{mthm3} not only proves the correspondence between solitary waves but also establishes a new existence result to \eqref{sta1} for $2 < p < 3$.
As we have seen above, the previous approaches \cite{AP1, APP, DM1, Wang} doesn't cover the case that $\omega > 0$  is less than but sufficiently close to $\bar{m}$. 
%In this gap, our result shows that \eqref{sta1} has a second radial positive solution, which blows up in $H^1(\R^3)$ norm.
In this respect, one family of solutions $u_c$ is actually not brand new because it is a simple consequence of implicit function theorem, which relies on nondegeneracy of the solution $u_\infty$.   
 However, the other family of solutions $v_c$ is brand new because $v_c$ bifurcates from a global minimizer of $I_\infty$, which blows up in $H^1$. As for the construction of $v_c$, it seems not easy to show whether the global minimum of $I_c$ is finite, unlike $I_\infty$.  
This prevents us from simply adopting the minimization argument.
To overcome this difficulty, we develop a new deformation argument, which strongly depends on the fact that the global minimum level of $I_\infty$ is bounded below. 
%To the best of our knowledge, this approach seems new in the literature.
We conjecture that if $c$ is sufficiently large, there exists a global minimizer of $I_c$, which converges to $v_\infty$.

We organize the paper as follows: In section \ref{prem}, we give variational settings for NSP and NMKG, and a simple proof  for the existence of a ground state to \eqref{2} for $3<p<6$. Section \ref{cons1} is devoted to construct  nontrivial solutions to \eqref{sta1} with the energy bound $E_\infty$ when $3 < p < 6$. 
In Section \ref{limm1}, we prove Theorem \ref{mthm1} by combining the results in Section \ref{cons1}.
In Section \ref{case23}, we deal with the case $2<p<3$. We construct two radial positive solutions of \eqref{sta1} and prove the convergence of their nonrelativistic limit.  
Finally, in Appendix, we give basic estimates, which are used in the proofs of main theorems.
\\ \\
{\bf Acknowledgement.}
This research of the second author was supported by Basic Science Research Program through the National Research Foundation of Korea(NRF) 
funded by the Ministry of Science and ICT (NRF-2020R1C1C1A01006415)

\section{Preliminaries}\label{prem}
This preliminary section introduces basic functional and variational settings for NMKG and NSP. 
In addition, we provide a simple proof for the existence of a ground state to \eqref{2} for every $3< p < 6$ and every $e,\bar{m}, \omega > 0$ such that $\bar{m} > \omega$.
\subsection{Function spaces}
The space $D^{1,2}(\R^3)$ is defined by the completion of $C_0^\infty(\R^3)$ with respect to the norm 
\[
\|u\|_{D^{1,2}(\R^3)}=\Big(\int_\Omega |\nabla u|^2 dx\Big)^{1/2}. 
\]
For an open set $\Omega\subset \R^3$ and $r\in [1,\infty)$, let us denote the norms
\[
\|u\|_{L^r(\Omega)}=\Big(\int_\Omega |u|^r dx\Big)^{1/r}, 
\quad \|u\|_{L^\infty(\Omega)}=\esssup_{x\in \Omega} |u(x)|,
\quad \|u\|_{H^1(\Omega)}=\Big(\int_\Omega |\nabla u|^2+u^2 dx\Big)^{1/2}. 
\]
We also use the following abbreviations, 
\[
\|u\|_{L^r}=\|u\|_{L^r(\R^3)}, \quad \|u\|_{D^{1,2}}=\|u\|_{D^{1,2}(\R^3)} \quad  \text{and} \quad \|u\|_{H^1}=\|u\|_{H^1(\R^3)}. 
\]
We denote by $H^1_r$ the Sobolev space of radial functions $u$ such that $u$, $\nabla u$ are in $L^2(\R^3)$.

\subsection{Variaional settings for NSP}
Recall the action functional for \eqref{a1},
\begin{equation*} 
\begin{aligned}
I_\infty(u)&=\frac12\int_{\R^3}|\nabla u|^2+2m\mu u^2+|\nabla \phi_u|^2 dx-\frac{1}{p}\int_{\R^3}|u|^pdx\\
&=\frac12\int_{\R^3}|\nabla u|^2+2m\mu u^2-qmu^2\phi_u dx-\frac{1}{p}\int_{\R^3}|u|^pdx.
\end{aligned}
\end{equation*}
The map $\lambda:u\in H^1\rightarrow \phi_u\in D^{1,2}$ is continuously differentiable, where $\phi_u$ satisfies \eqref{sta2} (see \cite{DM1}).
Since $\lambda^\prime(u) [v]$ satisfies
\[
-\Delta (\lambda^\prime(u) [v])=-2qm uv \mbox{ in } \R^3 \qquad \text{for } v \in H^1,
\]
we have
\[
\int_{\R^3}\nabla (\lambda^\prime(u) [v])\cdot \nabla \phi_u dx=-2qm\int_{\R^3}uv\phi_u dx. 
\]
Then we see that
\begin{align*}
I_\infty^\prime(u)v&= \int_{\R^3}\nabla u\cdot \nabla v+2m\mu uv+\nabla (\lambda^\prime(u) [v])\cdot \nabla \phi_u dx- \int_{\R^3}|u|^{p-2}uvdx\\
&=\int_{\R^3}\nabla u\cdot \nabla v+2m\mu uv-2qmuv \phi_u dx- \int_{\R^3}|u|^{p-2}uvdx,
\end{align*}
which shows that a critical point of $I_\infty$ is a weak solution to \eqref{a1}.
We define the Nehari and Pohozaev functionals for \eqref{a1} by 
\[
\begin{aligned}
&J_\infty(u)\equiv I_\infty^\prime(u)u=\int_{\R^3}|\nabla u|^2+2m\mu u^2-2qm u^2\phi_u-|u|^{p} dx, \\
&P_\infty(u) \equiv \int_{\R^3}\frac12|\nabla u|^2+3m\mu u^2-\frac52 qm u^2\phi_u-\frac{3}{p}|u|^{p} dx.
\end{aligned}
\]
We note that the values of $J_\infty$ and $P_\infty$ should be zero at every critical point of $I_\infty$  (see \cite{R}).
By defining $G_\infty(u)\equiv 2 J_\infty(u)-P_\infty(u)$, we denote
\begin{align*}
M_\infty \equiv \Big\{u\in H^1\setminus \{0\}\ &\Big|\ G_\infty(u)\equiv \int_{\R^3}\frac32|\nabla u|^2+m\mu u^2-\frac32 qm u^2\phi_u-\frac{2p-3}{p}|u|^{p} dx=0\Big\}
\end{align*}
and
\begin{equation}\label{a3}
E_\infty\equiv \inf_{u\in M_\infty}I_\infty(u).
\end{equation}
It is proved in \cite{R} that for $3 < p < 6$, $E_\infty$ equals to the ground energy level for \eqref{a1}, i.e.
\[
E_\infty = \inf\{ I_\infty(u) ~|~ u \neq 0,\, I_\infty'(u) = 0 \}.
\]

\subsection{Variational settings for NMKG}
The action functional for \eqref{sta1} is given by
\begin{align*}
I_c(u)&=\frac12\int_{\R^3}|\nabla u|^2+\Big(2m\mu-\frac{\mu^2}{c^2}\Big)u^2+|\nabla \Phi_u|^2+\Big(\frac{q}{c}\Big)^2u^2\Phi_u^2dx-\frac{1}{p}\int_{\R^3}|u|^pdx\\
&=\frac12\int_{\R^3}|\nabla u|^2+\Big(2m\mu-\frac{\mu^2}{c^2}\Big)u^2-q\Big(m-\frac{\mu}{c^2}\Big) u^2 \Phi_u dx-\frac{1}{p}\int_{\R^3}|u|^pdx.
\end{align*}The map $\Lambda:u\in H^1\rightarrow \Phi_u\in D^{1,2}$  is continuously differentiable, where $\Phi_u$ satisfies \eqref{sta3} (see \cite{DM1}). 
For $v\in H^1$, since $\Lambda^\prime(u) [v]$ satisfies
\[
-\Delta (\Lambda^\prime(u) [v])+\Big(\frac{q}{c}\Big)^2u^2 (\Lambda^\prime(u) [v])=-2\Big(\frac{q}{c}\Big)^2uv \Phi_u-2q\Big(m-\frac{\mu}{c^2}\Big)   uv,
\]
we have
\[
\int_{\R^3}\nabla (\Lambda^\prime(u) [v])\cdot \nabla \Phi_u+\Big(\frac{q}{c}\Big)^2u^2 (\Lambda^\prime(u) [v])\Phi_u dx=\int_{\R^3}-2\Big(\frac{q}{c}\Big)^2uv \Phi_u^2-2q\Big(m-\frac{\mu}{c^2}\Big)   uv\Phi_u dx.
\]
Then we see that for $v\in H^1$,
\begin{align*}
I_c^\prime(u)v&=\int_{\R^3}\nabla u\cdot \nabla v+\Big(2m\mu-\frac{\mu^2}{c^2}\Big)uv+\nabla \Phi_u\cdot \nabla (\Lambda^\prime(u) [v])+\Big(\frac{q}{c}\Big)^2 uv \Phi_u^2\\
&\qquad +\Big(\frac{q}{c}\Big)^2u^2\Phi_u (\Lambda^\prime(u) [v])-|u|^{p-2}uvdx\\
&=\int_{\R^3}\nabla u\cdot \nabla v+\Big(2m\mu-\frac{\mu^2}{c^2}\Big)uv-\Big(\frac{q}{c}\Big)^2 uv \Phi_u^2-2q\Big(m-\frac{\mu}{c^2}\Big)   u v\Phi_u-|u|^{p-2}uvdx.
\end{align*}
In particular, we have
\begin{align*}
J_c(u)\equiv I_c^\prime(u)u=\int_{\R^3}|\nabla u|^2+\Big(2m\mu-\frac{\mu^2}{c^2}\Big)u^2-\Big(\frac{q}{c}\Big)^2 u^2 \Phi_u^2-2q\Big(m-\frac{\mu}{c^2}\Big) u^2\Phi_u-|u|^{p} dx.
\end{align*}
For any critical point $w_c$ of $I_c$, it is clear that  $J_c(w_c) = 0$ and it is shown in \cite{DM} that the Pohozaev's identity $P_c(w_c)=0$ holds true,  where
\[
P_c(u)\equiv \int_{\R^3}\frac12 |\nabla u|^2+ \frac32\Big(2m\mu-\frac{\mu^2}{c^2}\Big)u^2-\frac{q^2}{c^2}\Phi_u^2u^2- \frac52q\Big(m-\frac{\mu}{c^2}\Big)u^2\Phi_u -\frac{3}{p}|u|^pdx.
\]
%We denote 
%\begin{equation}\label{a2}
%M_c \equiv \{u\in H^1\setminus \{0\}\ |\ J_c(u)\equiv I^\prime_c(u)u=0\} \quad \text{and} \quad E_c\equiv \inf_{u\in M_c}I_c(u).
%\end{equation}

%\begin{lem}
%Let $(u,\Phi)\in H^1\times D^{1,2}$ be a weak solution of \eqref{sta1}. Then there holds:
%\[
%  \int_{\R^3}\frac12 |\nabla u|^2+ \Big(3m\mu-\frac32 \frac{\mu^2}{c^2}\Big)u^2-\frac{3}{p}|u|^p-\Big(\frac52\frac{(mc^2-\mu)q}{c^2}+\frac{q^2}{c^2}\Phi\Big)\Phi u^2dx=0.
%\]
%\end{lem} 

\subsection{Existence of a ground state for $3 < p < 6$}
We recall the equation \eqref{2}, written as
\begin{equation}\label{geq1}
-\Delta u+\big(\bar{m}^2-(e\varphi_u+\omega)^2\big)u=|u|^{p-2}u \mbox{ in }\R^3
\end{equation}
where $e>0$, $0<\omega<\bar{m}$ and $\varphi_u$ is a unique solution of
\[
-\Delta \varphi+e^2\varphi u^2=-e\omega u^2.
\]
Here we point out that by the maximum principle, we have the uniform bound
\[
-\frac{\omega}{e}\le \varphi_u\le 0.
\]

\begin{prop}\label{gmmq1}
Assume that $3<p<6$, $e>0$ and $0<\omega<\bar{m}$. If there exists a non-trivial solution of \eqref{geq1}, then there exists a non-trivial ground state solution of \eqref{geq1}.
\end{prop}
\begin{proof}
Suppose that there exists a non-trivial solution solution of \eqref{geq1}.
We recall the action functional of \eqref{geq1}
\[
 I(u) =\frac12\int_{\R^3}|\nabla u|^2+(\bar{m}^2-\omega^2)u^2-e\omega \varphi_u u^2dx-\frac{1}{p}\int_{\R^3}|u|^pdx.
\]
and consider the minimization problem 
\[
\mathcal{S}=\inf\{ I(u)\ | \ u \in \mathcal{B}\},
\]
where
\[
\mathcal{B}\equiv \{ u \in H^1\ | \ u \mbox{ is a  non-trivial solution solution  of } \eqref{geq1}\}.
\] 
By the definition, a ground state solution $u$ of \eqref{geq1} is a nontrivial critical point of $I$ satisfying $I(u) = \mathcal{S}$.
Let us define
\[
\left\{\begin{aligned}
&T(u) \coloneqq I^\prime(u)u= \int_{\R^3}|\nabla u|^2+(\bar{m}^2-\omega^2)u^2-2e\omega \varphi_u u^2-e^2\varphi_u^2 u^2-|u|^pdx \\
&Q(u) \coloneqq \int_{\R^3}\frac12|\nabla u|^2+\frac32(\bar{m}^2-\omega^2)u^2-\frac52e\omega \varphi_u u^2-e^2\varphi_u^2 u^2-\frac{3}{p}|u|^pdx.
\end{aligned}\right.
\]
Since $T(v)=Q(v)=0$ for any $v \in \mathcal{B}$, (see \cite{DM}), one has
\begin{align*}
\frac{5p-12}{2}I(v)&=\frac{5p-12}{2}I(v)-T(v)+\frac{4-p}{2}Q(v) \\
&=\int_{\R^3}(p-3)|\nabla v|^2+\frac{p-2}{2}(\bar{m}^2-\omega^2)v^2+\frac{p-2}{2}e^2 v^2\varphi_{v}^2dx
\end{align*}
for $v\in \mathcal{B}$. This implies that $\mathcal{S}\ge 0$.

Let $\{u_n\}$ be a minimizing sequence of $\mathcal{S}$. From the estimates
\begin{equation}\label{mqq}
\begin{aligned}
\frac{5p-12}{2}\mathcal{S}+o(1) =\int_{\R^3}(p-3)|\nabla u_n|^2+\frac{p-2}{2}(\bar{m}^2-\omega^2)u_n^2+\frac{p-2}{2}e^2 u_n^2\varphi_{u_n}^2dx
\end{aligned}
\end{equation}
and
\begin{align*}
0=T(u_n)&= \int_{\R^3}|\nabla u_n|^2+(\bar{m}^2-\omega^2)u_n^2- \varphi_{u_n}(2e\omega +e^2\varphi_{u_n} )u_n^2-|u_n|^pdx\\
&\ge\int_{\R^3}|\nabla u_n|^2+(\bar{m}^2-\omega^2)u_n^2-|u_n|^pdx\ge C\|u_n\|_{L^p}^{2/p}-\|u_n\|_{L^p}^p,
\end{align*}
%where $C$ is a positive constant,
we deduce that $(u_n)$ is bounded in $H^1$ and $\|u_n\|_{L^p}\ge C_1$ for some  positive constant $C_1$. 
Then  we see from Lemma 1.1 in \cite{L},
\begin{equation*}
\sup_{x\in \R^3}\int_{B_1(x)}|u_n|^2dx=\int_{B_1(x_n)}|u_n|^2dx\ge C_2>0,
\end{equation*}
where $x_n\in \R^3$ and $C_2$ is a positive constant. Then we may assume that $u_n(\cdot+x_n)$ converges to $u\not\equiv0$ weakly in $H^1$. It is standard to show that $u$ is a non-trivial critical point of $I$. Moreover, by \eqref{mqq} and the fact that $u$ is a non-trivial critical point of $I$, we see that
\begin{align*}
\frac{5p-12}{2}\mathcal{S}&=\liminf_{n\rightarrow \infty}\int_{\R^3}(p-3)|\nabla u_n|^2+\frac{p-2}{2}(\bar{m}^2-\omega^2)u_n^2+\frac{p-2}{2}e^2 u_n^2\varphi_{u_n}^2dx\\
&\ge \int_{\R^3}(p-3)|\nabla u|^2+\frac{p-2}{2}(\bar{m}^2-\omega^2)u^2+\frac{p-2}{2}e^2 u^2\varphi_{u}^2dx =\frac{5p-12}{2}I(u),
\end{align*}
which implies that $u$ is a non-trivial ground state  solution of \eqref{geq1}.
\end{proof}
Observe that Proposition \ref{gmmq1} implies the existence of a ground state to \eqref{geq1} for any $e, \bar{m}, \omega > 0$ such that $0 < \omega < \bar{m}$ since
there exists a nontrivial solution at those ranges of parameters by \cite{APP}.

\section{ Construction of nontrivial solutions to NKGM with the energy bound $E_\infty$}\label{cons1}
In this section, based on the idea of \cite{BJ}, we shall construct a family of nontrivial solutions $w_c$ to \eqref{sta1} satisfying
\[
\limsup_{c\to\infty}I_c(w_c) \leq E_{\infty}.
\]
Before proceeding further, we first introduce a modified functional $\tilde{I}_c$ as
\begin{align*}
\tilde{I}_c(u)=\frac12\int_{\R^3}|\nabla u|^2+\Big(2m\mu-\frac{\mu^2}{c^2}\Big)u^2-q\Big(m-\frac{\mu}{c^2}\Big) u ^2 \Phi_{u} dx-\frac{1}{p}\int_{\R^3}u_+^pdx,
\end{align*}
where $c>0$ and  $u_+=\max\{u,0\}$.
A critical point of $\tilde{I}_c$ corresponds to a solution of 
\begin{equation}\label{qq1}
\begin{aligned}
-\Delta u+&\Big(2m\mu-\Big(\frac{\mu}{c}\Big)^2\Big)u-\Big(\frac{q}{c}\Big)^2u \Phi^2-2q\Big(m-\frac{\mu}{c^2}\Big)u\Phi -u_+^{p-1}=0 \mbox{ in }\R^3,\\
&-\Delta \Phi+\frac{q^2}{c^2}u^2\Phi=-q\Big(m-\frac{\mu}{c^2}\Big)   u^2 \mbox{ in } \R^3.
\end{aligned}
\end{equation}
It is possible to show from the maximum principle that a critical point $u$ of $\tilde{I}_c$ is positive everywhere in $\R^3$ for $c\ge \sqrt{\frac{2m}{\mu}}$.
Indeed, since $-\frac{c^2}{q}\Big(m-\frac{\mu}{c^2}\Big)\le \Phi_u \le0 $, multiplying $u_-$ to the equation
$$-\Delta u+ \Big(2m\mu-\big(\frac{\mu}{c}\big)^2\Big)u-\Big(\frac{q}{c}\Big)^2u \Phi_u^2-2q\Big(m-\frac{\mu}{c^2}\Big)u\Phi_u -u_+^{p-1}=0 \mbox{ in }\R^3
$$
 and then integrating over $\R^3$, we have
\begin{align*}
&\int_{\R^3}|\nabla u_-|^2+\Big(2m\mu-\frac{\mu^2}{c^2}\Big)u_-^2dx\\
&\le \int_{\R^3}|\nabla u_-|^2+\Big(2m\mu-\frac{\mu^2}{c^2}\Big)u_-^2-u_-^2 \Phi_u\Big[\Big(\frac{q}{c}\Big)^2 \Phi_u+2q\Big(m-\frac{\mu}{c^2}\Big) \Big] dx=0,
\end{align*}
where $u_-=\min\{u,0\}$.
Therefore a nontrivial critical point of $\tilde{I}_c$ gives a positive solution to \eqref{sta1}.
We also define
 \begin{align*}
&\tilde{I}_\infty(u) \coloneqq \frac12\int_{\R^3}|\nabla u|^2+2m\mu u^2-qmu ^2\phi_{u} dx-\frac{1}{p}\int_{\R^3}u_+^pdx, \\
&\tilde{J}_\infty(u) \coloneqq I_\infty^\prime(u)u=\int_{\R^3}|\nabla u|^2+2m\mu u^2-2qm u ^2\phi_{u}-u_+^{p} dx, \\
&\tilde{P}_\infty(u) \coloneqq \int_{\R^3}\frac12|\nabla u|^2+3m\mu u^2-\frac52 qm u ^2\phi_{u}-\frac{3}{p}u_+^{p} dx.
\end{align*}

Let $\mathcal{A}\equiv \{u\in H^1\ | \  \tilde{I}_\infty^\prime(u)=0, \tilde{I}_\infty(u)=E_\infty, \mbox{ and } \max_{\R^3}u=u(0)\}.$ We note that $\mathcal{A}\neq \emptyset$. Indeed, if $u\in \mathcal{M}_\infty$ satisfies $I_\infty(u)=E_\infty$,  we see that $|u|$ satisfies $\tilde{I}_\infty(|u|)=E_\infty$ and $\tilde{I}_\infty^\prime(|u|)=0$. \begin{prop}\label{arq1}
For $3< p<6$, there exist positive constants $C_1$ and $C_2$ independent of $U\in \mathcal{A}$ such that  for $U\in \mathcal{A}$,
\[
U(x)+|\nabla U(x)|\le C_1 \exp(-C_2|x|).  
\]
Moreover, $\inf_{U\in \mathcal{A}}\|U\|_{L^\infty}>0$.
\end{prop}
\begin{proof}
Let  $U\in \mathcal{A}$. It follows from 
\begin{equation}\label{tt1}
\begin{aligned}
E_\infty&=\tilde{I}_\infty(U)=\tilde{I}_\infty(U)-\frac{2}{5p-12}\tilde{J}_\infty(U)-\frac{p-4}{5p-12}\tilde{P}_\infty(U)\\
&=\int_{\R^3}\frac{2(p-3)}{5p-12}|\nabla U|^2+\frac{2(p-2)}{5p-12}m\mu U^2dx 
\end{aligned}
\end{equation}where $U\in \mathcal{A},$
that $\mathcal{A}$ is bounded in $H^1$ if $3< p<6$. Then, since 
\begin{align*}
\|\phi_{U}+|U|^{p-2}\|_{L^\frac{6}{p-2}(\Omega)}&\le \|\phi_{U}\|_{L^\frac{6}{p-2}(\Omega)}+\|U\|_{L^6(\Omega)}^{p-2}\le |\Omega|^{\frac{p-2}{6}-\frac{1}{6}}\|\phi_{U}\|_{L^6(\Omega)}+\|U\|_{L^6(\Omega)}^{p-2}\\
& \le C\big(|\Omega|^{\frac{p-2}{6}-\frac{1}{6}}\|U\|_{H^1}^2 +\|U\|_{H^1}^{p-2}\big),
\end{align*}where  $3<p<6$, $U\in \mathcal{A}$, $\Omega$ is a bounded domain in $\R^3$ and $C$ is a positive constant independent of $U\in \mathcal{A}$,
we see that $\mathcal{A}$ is bounded in $L^\infty$ (see \cite[Theorem 4.1]{HL}).

We claim that $\lim_{|x|\rightarrow \infty} U(x)=0$ uniformly for $U\in \mathcal{A}$. Indeed, contrary to our claim, suppose that there exist $\{U_i\}_{i=1}^\infty\subset \mathcal{A}$ and $\{x_i\}_{i=1}^\infty\subset \R^N$ satisfying $\lim_{i\rightarrow \infty}|x_i|=\infty$ and $\liminf_{i\rightarrow \infty}U_i(x_i)>0$. Denote $V_i\equiv U_i(\cdot+x_i)$. We note that if  $u_i\rightharpoonup u$ in $H^1$,  $\phi_{u_i}\rightharpoonup \phi_{u}$ in $D^{1,2}$. Then  if  $u_i\rightharpoonup u$ in $H^1$,  for $\psi\in C_0^\infty(\R^3)$, 
\begin{equation}\label{a131}
\int_{\R^3}(u_i \phi_{u_i}-u\phi_{u})\psi dx=\int_{\R^3}(u_i-u) \phi_{u_i}\psi+u(\phi_{u_i}-\phi_{u})\psi dx=o(1)
\end{equation}
as $i\rightarrow \infty$.
By \eqref{a131} and the fact that $\{U_i, V_i\}_{i=1}^\infty$ is bounded in $H^1$, we see that $U_i$ and $V_i$ converge to $U$ and $V$ weakly in $H^1$ as $i\rightarrow \infty$ , up to a subsequence, respectively, where $U$ and $V$ are non-trivial solutions of \eqref{a1}. It follows from \eqref{tt1}
that for $2R\le |x_i|$,
\begin{equation}\label{eee1q}
\begin{aligned}
E_\infty&=\liminf_{i\rightarrow \infty}\tilde{I}_\infty(U_i)=\liminf_{i\rightarrow \infty}\int_{\R^3}\frac{2(p-3)}{5p-12}|\nabla U_i|^2+\frac{2(p-2)}{5p-12}m\mu U_i^2dx\\
&\ge\liminf_{i\rightarrow \infty}\int_{B(0,R)}\frac{2(p-3)}{5p-12}|\nabla U_i|^2 +\frac{2(p-2)}{5p-12}m\mu U_i^2dx\\
&\qquad+\liminf_{i\rightarrow \infty}\int_{B(x_i,R)}\frac{2(p-3)}{5p-12}|\nabla U_i|^2+\frac{2(p-2)}{5p-12}m\mu U_i^2dx\\
&\ge \int_{B(0,R)}\frac{2(p-3)}{5p-12}|\nabla U |^2 +\frac{2(p-2)}{5p-12}m\mu U ^2dx\\
&\qquad+\int_{B(0,R)}\frac{2(p-3)}{5p-12}|\nabla V|^2+\frac{2(p-2)}{5p-12}m\mu V^2dx.
\end{aligned}
\end{equation}
 Since
\[
\tilde{I}_\infty(U), \tilde{I}_\infty(V)\ge \tilde{I}_\infty(W) \mbox{ for any } W\in \mathcal{A},
\]
if we take large $R>0$ in \eqref{eee1q}, we deduce a contradiction. This implies that $\lim_{|x|\rightarrow \infty}U(x)=0$ uniformly for $U\in \mathcal{A}$.

We note that for large $|x|$,
\begin{equation*}
\begin{aligned}
\phi_{U}(x)&=-\frac{qm}{4\pi}\int_{\R^3}\frac{U ^2(y)}{|x-y|}dy =-\frac{qm}{4\pi}\int_{B(x,R)}\frac{U^2(y)}{|x-y|}dy-\frac{qm}{4\pi}\int_{\R^3\setminus B(x,R)}\frac{U ^2(y)}{|x-y|}dy\\
&=o(1)R^2+O(1)\frac{1}{R}=o(1) 
\end{aligned}
\end{equation*}
uniformly in $U\in \mathcal{A}$. Then, by the comparison principle and the elliptic estimates, we see that for $U\in \mathcal{A}$,
\[
U(x)+|\nabla U(x)|\le C_1 \exp(-C_2|x|),  
\]
where  $C_1$ and $C_2$ are positive constants  independent of $U\in \mathcal{A}$. 

To show $\inf_{U\in \mathcal{A}}\|U\|_{L^\infty}>0$, we assume that there exists $\{U_i\}_{i=1}^\infty\subset \mathcal{A}$ such that $\|U_i\|_{L^\infty}\rightarrow 0$ as $i\rightarrow \infty$.  Then, since $U_i$ satisfies
\[
-\Delta U_i+2m\mu U_i-U_i^{p-1}\le -\Delta U_i+ 2m\mu U_i -2qm U_i \phi_{U_i}-U_i^{p-1}=0 \mbox{ in }\R^3,
\]
we see that $\|U_i\|_{H^1}\rightarrow  0$ as $i\rightarrow \infty$, which is a contradiction to \eqref{tt1}.

\end{proof}
For a fixed $U_0\in \mathcal{A}$, we define $\gamma(t)(x)=t^2U_0(tx)$. It follows from
\[
\tilde{I}_\infty(\gamma(t))=\frac12\int_{\R^3}t^3|\nabla U_0|^2+2m\mu tU_0^2-qmt^3U_0 ^2\phi_{U_0 } dx-\frac{t^{2p-3}}{p}\int_{\R^3}U_0^pdx
\]
that for $3<p<6$, there exists $t_0>1$ such that $\tilde{I}_\infty(\gamma(t))<0$ for $t\ge t_0$. Moreover, by \cite[Lemma 3.3]{R} and the fact that $U_0$ is a critical point of $\tilde{I}_\infty$, we see that for $3<p<6$, $t=1$ is a unique critical point of $\tilde{I}_\infty(\gamma(t))$, corresponding to its maximum.

We define 
\[
\hat{e}_c \coloneqq \max_{t\in[0,t_0]}\tilde{I}_c(\gamma(t)), \quad \text{and} \quad e_c\coloneqq \inf_{\Gamma\in \mathcal{W}}\max_{s\in[0,1]}\tilde{I}_c(\Gamma(s)),
\]
where $\mathcal{W}\equiv\{\Gamma\in C([0,1],H^1)\ | \ \Gamma(0)=0, \Gamma(1)=\gamma(t_0)\}$. 

%Then, by \eqref{aq4}, we have $\limsup_{c\rightarrow \infty} \hat{e}_c\le E_\infty$
\begin{prop}\label{prop-upper-estimate} 
Let $3<p<6$. Then we have
\[
\limsup_{c\rightarrow \infty} \hat{e}_c\le E_\infty.\]
\end{prop}
\begin{proof}
We see  from Lemma \ref{yyy1} and the scaling $\phi_{t^2U_0(t\cdot)}=t^2\phi_{U_0}(t\cdot),$ that for $t\in[0,t_0]$,
\begin{equation}\label{aq4}
\begin{aligned}
\tilde{I}_c(\gamma(t))&=\frac12\int_{\R^3}|t^3(\nabla U_0)(tx)|^2+\Big(2m\mu-\frac{\mu^2}{c^2}\Big)t^4U_0^2(tx)-q\Big(m-\frac{\mu}{c^2}\Big)t^4 U_0^2(t x) \Phi_{t^2U_0(t\cdot)} dx\\
&\qquad -\frac{t^{2p}}{p}\int_{\R^3}(U_0(tx))^pdx\\
&=\frac12\int_{\R^3}|t^3(\nabla U_0)(tx)|^2+ 2m\mu t^4U_0^2(tx)-q m t^4 U_0^2(t x) \phi_{t^2U_0(t\cdot)} dx\\
&\qquad  -\frac{t^{2p}}{p}\int_{\R^3}(U_0(tx))^pdx+o(1)\\
&= \frac12 \int_{\R^3}t^3|\nabla U_0|^2+2m\mu tU_0^2-qmt^3U_0^2\phi_{U_0} dx-\frac{t^{2p-3}}{p}\int_{\R^3}(U_0)^pdx+o(1)\\
&=\tilde{I}_\infty(\gamma(t))+o(1),
\end{aligned}
\end{equation}
where $o(1)$ is uniform in $t\in[0,t_0]$ as $c\rightarrow \infty$. 
Thus, since $t=1$ is a unique maximum point of $\tilde{I}_\infty(\gamma(t))$  for $3<p<6$,  we deduce that  
\[
\hat{e}_c = \max_{s\in[0,1]}\tilde{I}_c(\gamma(t_0s))=\tilde{I}_\infty(U_0)+o(1) =E_\infty+o(1)
\]as $c\rightarrow \infty$.
\end{proof}

\begin{prop}\label{path1}
Let $3<p<6$. Then we have
\[
\liminf_{c\rightarrow \infty}e_c\ge E_\infty.\]
\end{prop}
\begin{proof}
We note that for $\Gamma\in \mathcal{W}$,
\begin{align*}
\tilde{I}_c(\Gamma(t))&=\frac12\int_{\R^3}|\nabla \Gamma(t)|^2+ 2m\mu\Gamma^2(t)-qm \Gamma^2(t) \phi_{\Gamma(t)} dx-\frac{1}{p}\int_{\R^3}(\Gamma(t))_+^pdx\\
&\qquad -\frac{1}{c^2}\int_{\R^3}\mu^2\Gamma^2(t)- q\mu \Gamma^2(t) \Phi_{\Gamma(t)} dx-\frac12 qm\int_{\R^3}\Gamma^2(t)(\Phi_{\Gamma(t)} -\phi_{\Gamma(t)})dx\\
&=\tilde{I}_\infty(\Gamma(t))+G_c(t),
\end{align*}
where $G_c(t)\equiv  -\frac{1}{c^2}\int_{\R^3}\mu^2\Gamma^2(t)- q\mu \Gamma^2(t) \Phi_{\Gamma(t)} dx-\frac12 qm\int_{\R^3}\Gamma^2(t)(\Phi_{\Gamma(t)} -\phi_{\Gamma(t)})dx$. By Lemma \ref{yyy1},
we have
\begin{align*}
|G_c(t)|=o(1) \mbox{ uniformly in } t\in [0,1] \mbox{ as } c\rightarrow \infty.
\end{align*}
 Then, since
\[
\max_{t\in[0,1]}\tilde{I}_\infty(\Gamma(t))\ge E_\infty,
\]
where $\Gamma\in \mathcal{W}$ (see \cite[Lemma 2.4]{AP}),
 we have
\begin{align*}
e_c&\ge E_\infty+\inf_{\Gamma\in \mathcal{W}}\max_{t\in[0,1]}G_c(t) \ge  E_\infty-\inf_{\Gamma\in \mathcal{W}}\max_{t\in[0,1]}|G_c(t)|= E_\infty+o(1)
\end{align*}
as $c\rightarrow \infty$.
\end{proof}

We define
\[
\mathcal{X} \equiv \{U(\cdot-y)\ | \ U\in \mathcal{A}, y\in \R^3\}
\]
and 
\[
N_d(\mathcal{X})\equiv \{u\in H^1\ | \ \inf_{v\in \mathcal{X}}\|u-v\|_{H^1}\le d\},
\]
where   $d>0$ is a constant  and $\mathcal{A}\equiv \{u\in H^1\ | \ \tilde{I}_\infty^\prime(u)=0, \tilde{I}_\infty(u)=E_\infty, \mbox{ and } \max_{\R^3}u=u(0)\}.$
\begin{prop}\label{annul1}
Let $3<p<6$.   For large $c>0$, for small $d>0$, and for any $d^\prime\in (0,d)$, there exists $\nu\equiv \nu(d,d^\prime)>0$ independent of $c>0$ such that
\[
\inf\{\|\tilde{I}_c^\prime(u)\|_{H^{-1}}\ | \ \tilde{I}_c(u)\le \hat{e}_c, u\in N_d(\mathcal{X})\setminus N_{d^\prime}(\mathcal{X})\}\ge \nu>0.
\]
\end{prop}
\begin{proof}
Let $\{c_i\}_{i=1}^\infty$ be such that $\lim_{i\rightarrow \infty}c_i=\infty$. It suffices to show that for small $d>0$,  if
\begin{align*}
u_{c_i}\in N_{d}(\mathcal{X}),\ \ \tilde{I}_{c_i}(u_{c_i})\le \hat{e}_{c_i},\ \mbox{ and }\ \|\tilde{I}^\prime_{c_i}(u_{c_i})\|_{H^{-1}}\rightarrow 0 
\end{align*}
as $i\rightarrow \infty$, then 
\[
\inf_{v\in \mathcal{X}}\|u_{c_i}-v\|_{H^1}\rightarrow 0 \mbox{ as } i\rightarrow \infty.
\]
 For the sake of simplicity of notation, we write $c$ for $c_i$. Since $u_c\in  N_{d}(\mathcal{X})$, we have
\begin{equation}\label{appo}
\|u_c(x)-U_c(x-y_c)\|_{H^1}\le d,
\end{equation}
where $U_c\in \mathcal{A}$ and $y_c\in \R^3$. We define $\eta\in C_0^\infty(\R^3)$ such that $0\le \eta\le1$, $\eta(x)=1$ for $|x|\le 1$, $\eta(x)=0$ for $|x|\ge 2$, and $|\nabla \eta|\le 2$. Also, we set  $\tilde{\eta}_c(x)=\eta(\frac{x-y_c}{c})$. We divide the proof into three steps.\\
\noindent {\bf Step 1.} $ \tilde{I}_c(u_c)\ge \tilde{I}_\infty(v_c)+\tilde{I}_\infty(w_c)+o(1)$ as $c\rightarrow \infty$, where $v_c =\tilde{\eta}_cu_c$ and $w_c=(1-\tilde{\eta}_c)u_c$. 

We claim first that for $\alpha\in (2,6)$,
\begin{equation}\label{arq4}
\lim_{c\rightarrow \infty}\int_{B(y_c,2c)\setminus B(y_c,c)}|u_c|^{\alpha}dx=0.
\end{equation}
Suppose that there exist  $z_c\in B(y_c,2c)\setminus B(y_c,c)$ and $R>0$ such that 
\begin{equation}\label{arq3}
\liminf_{c\rightarrow \infty}\int_{B(z_c,R)}|u_c|^2dx>0.
\end{equation}
Denote $\tilde{u}_c=u_c(\cdot+z_c)$. We note that, by Lemma \ref{yyy1} and the fact that $\|u_c\|_{H^1}$ is bounded, for $\psi\in C_0^\infty(\R^3)$,
\begin{equation}\label{bqqwq1}
\begin{aligned}
&\tilde{I}_c^\prime(\tilde{u}_c)\psi  \\
&=\int_{\R^3}\nabla \tilde{u}_c\cdot \nabla \psi+\Big(2m\mu-\frac{\mu^2}{c^2}\Big)\tilde{u}_c \psi-\Big(\frac{q}{c}\Big)^2 \tilde{u}_c \psi \Phi_{\tilde{u}_c}^2-2q\Big(m-\frac{\mu}{c^2}\Big)   \tilde{u}_c \psi\Phi_{\tilde{u}_c}-(\tilde{u}_c)_+^{p-1} \psi dx\\
&=\int_{\R^3}\nabla \tilde{u}_c\cdot \nabla \psi+2m\mu\tilde{u}_c \psi-2qm   \tilde{u}_c \psi\phi_{\tilde{u}_c}-(\tilde{u}_c)_+^{p-1} \psi dx\\
&\qquad + \int_{\R^3}-\frac{\mu^2}{c^2}\tilde{u}_c \psi-\Big(\frac{q}{c}\Big)^2 \tilde{u}_c \psi \Phi_{\tilde{u}_c}^2+2q\frac{\mu}{c^2} \tilde{u}_c \psi\Phi_{\tilde{u}_c}-2qm   \tilde{u}_c \psi(\Phi_{\tilde{u}_c}-\phi_{\tilde{u}_c})dx\\
&=\tilde{I}_\infty^\prime(\tilde{u}_c)\psi+o(1)
\end{aligned}
\end{equation}
as $c\rightarrow \infty$.
By \eqref{a131} and the assumption that $\|\tilde{I}^\prime_{c}(u_{c})\|_{H^{-1}}\rightarrow 0$ as $c\rightarrow \infty$ , we have $u_c(\cdot+z_c)\rightharpoonup \tilde{U}\not\equiv0$ in $H^1$, where $\tilde{U}$ satisfies $\tilde{I}_\infty^\prime(\tilde{U})=0$. By \eqref{tt1}, we have
\begin{equation}\label{arq2}
\int_{\R^3}|\nabla \tilde{U}|^2+\tilde{U}dx\ge  E_\infty\Big(\max\Big\{\frac{2(p-3)}{5p-12},\frac{2(p-2)}{5p-12}m\mu\Big\}\Big)^{-1}.
\end{equation}
 Then, by Proposition \ref{arq1} and the fact that $|z_c-y_c|\ge c$, we see that for $R>0$,
\begin{align*}
d^2&\ge \|u_c(x)-U_c(x-y_c)\|_{H^1}^2 =\|\tilde{u}_c(x)-U_c(x+z_c-y_c)\|_{H^1}^2\\
&\ge \|\tilde{u}_c(x)-U_c(x+z_c-y_c)\|_{H^1(B(0,R))}^2 =\|\tilde{u}_c(x)\|_{H^1(B(0,R))}^2+o(1)\ge \|\tilde{U}\|_{H^1(B(0,R))}^2
\end{align*}
as $c\rightarrow \infty$. If we take small $d>0$, by \eqref{arq2}, we deduce a contradiction. Since there does not exists such  a sequence $\{z_c\}$ satisfying \eqref{arq3}, by \cite[Lemma 1.1]{L}, we deduce \eqref{arq4}. Then, by \eqref{arq4}, we have
\begin{equation}\label{arq5}
\int_{\R^3}(u_c)_+^p-(v_c)_+^p-(w_c)_+^pdx=o(1)
\end{equation} as $c\rightarrow \infty$,
where $v_c$ and $w_c$ are given in \eqref{arq4} above.  By \eqref{arq4} and Lemma \ref{iiop},
\begin{align*}
\int_{B(y_c,2c)\setminus B(y_c,c)}u_c^2|\phi_{u_c}|dx&\le \|\phi_{u_c}\|_{L^6(B(y_c,2c)\setminus B(y_c,c))}\|u_c^2\|_{L^{6/5}(B(y_c,2c)\setminus B(y_c,c))}\\
&\le C_1\|u_c\|_{H^1}^2\|u_c\|_{L^{12/5}(B(y_c,2c)\setminus B(y_c,c))}^2\rightarrow 0 
\end{align*}
as $c\rightarrow \infty$, where $C_1$ is a positive constant. From this and the fact that $|\nabla \eta_c|\le 2/c$, we see that
\begin{equation}\label{fi1}
\begin{aligned}
&\int_{\R^3}v_c^2 \phi_{v_c}+w_c^2 \phi_{w_c}-u_c^2 \phi_{u_c}dx\\
&=\int_{ B(y_c,c)\cup (\R^3\setminus  B(y_c,2c) )}v_c^2 \phi_{v_c}+w_c^2 \phi_{w_c}-u_c^2 \phi_{u_c}dx+o(1)\\
&=\frac{qm}{4\pi}\int_{B(y_c,c)\cup (\R^3\setminus  B(y_c,2c) )}\int_{\R^3}\frac{u_c^2(x)u_c^2(y)-v_c^2(x)v_c^2(y)-w_c^2(x)w_c^2(y)}{|x-y|}dydx+o(1)\\
&=\frac{qm}{4\pi}\int_{B(y_c,c)}\int_{\R^3}\frac{u_c^2(x)(u_c^2(y)-v_c^2(y))}{|x-y|}dydx\\
&\qquad +\frac{qm}{4\pi}\int_{\R^3\setminus B(y_c,2c)}\int_{\R^3}\frac{u_c^2(x)(u_c^2(y)-w_c^2(y))}{|x-y|}dydx+o(1)\ge o(1) 
\end{aligned}
\end{equation}
as $c\rightarrow \infty$. Thus, by \eqref{arq5}, \eqref{fi1}, Lemma \ref{yyy1} and the fact that $|\nabla \eta_c|\le 2/c$, 
we have
\begin{align*}
\tilde{I}_c(u_c)&=\frac12\int_{\R^3}|\nabla u_c|^2+ 2m\mu u_c^2-qm u_c^2 \phi_{u_c} dx-\frac{1}{p}\int_{\R^3}(u_c)_+^pdx\\
&\qquad -\frac{1}{2c^2}\int_{\R^3}\mu^2u_c^2- q\mu u_c^2 \Phi_{u_c} dx-\frac12 qm\int_{\R^3}u_c^2(\Phi_{u_c} -\phi_{u_c})dx\\
&\ge \tilde{I}_\infty(v_c)+\tilde{I}_\infty(w_c)+\int_{\R^3}\nabla v_c\cdot \nabla w_c+2m\mu v_c w_cdx+o(1)\\
&=\tilde{I}_\infty(v_c)+\tilde{I}_\infty(w_c)+\int_{\R^3}(1-\tilde{\eta}_c)\tilde{\eta}_c|\nabla u_c|^2+2m\mu (1-\tilde{\eta}_c)\tilde{\eta}_cu_c^2dx+o(1)\\
&\ge \tilde{I}_\infty(v_c)+\tilde{I}_\infty(w_c)+o(1)
\end{align*}
as $c\rightarrow \infty$.

\noindent{\bf Step 2.} $\tilde{I}_\infty(w_c)\ge 0$ for large $c$, where  $w_c=(1-\tilde{\eta}_c)u_c$.

We note that, by  Lemma \ref{iiop},
\begin{align*}
\Big|\int_{\R^3} w_c^2 \phi_{w_c} dx\Big|&\le \|\phi_{w_c}\|_{L^6}\| w_c^2\|_{L^{6/5}} \le C_2\|w_c\|_{H^1}^4,
\end{align*}
where $C_2$ is a positive constant independent of $c$. Moreover, by \eqref{appo} and  Proposition \ref{arq1}, $\|w_c\|_{H^1}\le 2d$ for large $c>0$. Then we have
\begin{equation}\label{azi1}
\begin{aligned}
\tilde{I}_\infty(w_c)&=\frac12\int_{\R^3}|\nabla w_c|^2+ 2m\mu w_c^2-qm w_c^2 \phi_{w_c} dx-\frac{1}{p}\int_{\R^3}(w_c)_+^pdx\\
&\ge \|w_c\|_{H^1}^2\Big(\min\Big\{\frac12,m\mu\Big\}-qmC_2(\|w_c\|_{H^1}^2+\|w_c\|_{H^1}^{p-2})\Big).
\end{aligned}
\end{equation} Taking $d>0$ small, we deduce that $\tilde{I}_\infty(w_c)\ge 0$ for large $c$.

\noindent{\bf Step 3.} $v_c\rightarrow \tilde{V}(\cdot-z)$ in $H^1$, where $\tilde{V}\in \mathcal{A}$, $z\in \R^3$ and $v_c=\tilde{\eta}_cu_c$.

Let $W_c\equiv v_c(\cdot+y_c)$. We can assume that $W_c\rightharpoonup W\not\equiv 0$ in $H^1$, up to a subsequence, as $c\rightarrow \infty.$ Since $W_c-u_c(\cdot+y_c)\rightharpoonup0$ in $H^1$, $\phi_{W_c}-\phi_{u_c(\cdot+y_c)}\rightharpoonup0$ in $D^{1,2}$. Then for any $\psi \in C_0^\infty(\R^3)$,
\begin{align*}
 \int_{\R^3}(W_c\phi_{W_c}-u_c(\cdot+y_c)\phi_{u_c(\cdot+y_c)})\psi dx &= \int_{\R^3} (W_c-W)\big(\phi_{W_c}-\phi_{u_c(\cdot+y_c)}\big)\psi +W\big(\phi_{W_c}-\phi_{u_c(\cdot+y_c)}\big)\psi \\
&\qquad +(W_c -u_c(\cdot+y_c))\phi_{u_c(\cdot+y_c)} \psi dx\rightarrow 0
\end{align*}
as $c\rightarrow \infty$. From this, \eqref{a131}, \eqref{bqqwq1} and the assumption that $\|\tilde{I}^\prime_{c}(u_{c})\|_{H^{-1}}\rightarrow 0$ as $c\rightarrow \infty$, we can see that $W$ satisfies $\tilde{I}^\prime_\infty(W)=0$. By the maximum principle, $W$ is positive. Suppose that there exist $R>0$ and a sequence $\tilde{z}_c\in B(y_c,2c)$ satisfying
\[
\liminf_{c\rightarrow \infty}|\tilde{z}_c-y_c|=\infty \mbox{ and } \liminf_{c\rightarrow \infty}\int_{B(\tilde{z}_c,R)}|v_c|^2dx>0.
\]
Then $v_c(\cdot+z_c)$ converges weakly to $\tilde{W}$ in $H^1$, where $I_\infty^\prime(\tilde{W})=0$. By the same arguments in Step 1, we deduce a contradiction. By \cite[Lemma 1.1]{L}, we have
\begin{equation}\label{non1}
\lim_{c\rightarrow \infty}\int_{\R^3}(W_c)_+^pdx=\int_{\R^3}W^pdx.
\end{equation}
We note that
\begin{equation}\label{rez1}
\begin{aligned}
\liminf_{c\rightarrow \infty}\Big(- \int_{\R^3}W_c^2 \phi_{W_c} dx\Big)&=\liminf_{c\rightarrow \infty}\int_{\R^3}\int_{\R^3}\frac{W_c^2(x)W_c^2(y)}{|x-y|}dydx\\
&\ge \int_{\R^3}\int_{\R^3}\frac{W ^2(x)W ^2(y)}{|x-y|}dydx =-\int_{\R^3}W^2 \phi_{W} dx.
\end{aligned}
\end{equation}
Then, by \eqref{non1}, \eqref{rez1} and Lemma \ref{yyy1}, we have
\begin{equation}\label{rez}
\begin{aligned}
\liminf_{c\rightarrow \infty}\tilde{I}_\infty(W_c)&=\liminf_{c\rightarrow \infty} \frac12\int_{\R^3}|\nabla W_c|^2+ 2m\mu W_c^2-qm W_c^2 \phi_{W_c} dx-\frac{1}{p}\int_{\R^3}(W_c)_+^pdx\\
&\ge \tilde{I}_\infty(W).
\end{aligned}
\end{equation}
By \eqref{rez}, the results of Step1 and Step 2, and the assumption that $\tilde{I}_c(u_c)\le \hat{e}_c$, we see that $\tilde{I}_\infty(W)=E_\infty.$  By \eqref{non1}, \eqref{rez1} and \eqref{rez},  we have 
\begin{align*}
\limsup_{c\rightarrow \infty} &\int_{\R^3}|\nabla W_c|^2+ 2m\mu W_c^2-qm W_c^2 \phi_{W_c} dx =\int_{\R^3}|\nabla W|^2+ 2m\mu W^2-qm W^2 \phi_{W} dx\\
&\le \int_{\R^3}|\nabla W|^2+ 2m\mu W^2+\limsup_{c\rightarrow\infty}\Big(-\int_{\R^3}qm W_c^2 \phi_{W_c} dx\Big),
\end{align*}
which implies that $W_c\rightarrow W$ in $H^1$. By \eqref{azi1}, the result of Step 1 and the fact that $\hat{e}_c\rightarrow E_\infty$, we have for small $d>0$,
\begin{align*}
\hat{e}_c\ge \tilde{I}_c(u_c)&\ge \tilde{I}_\infty(v_c)+\frac12\min\Big\{\frac12,m\mu\Big\} \|w_c\|_{H^1}^2+o(1)\ge E_\infty+\frac12\min\Big\{\frac12,m\mu\Big\} \|w_c\|_{H^1}^2+o(1) 
\end{align*}  
as $c\rightarrow \infty$, which implies that $\|w_c\|_{H^1}\rightarrow 0$ as $c\rightarrow \infty$. Thus, letting $W=\tilde{V}(\cdot-z)$, where $\tilde{V}\in \mathcal{A}$ and $z\in \R^3$, we have
\[
\|u_c-\tilde{V}(\cdot-y_c-z)\|_{H^1}\le \|v_c(\cdot+y_c)-\tilde{V}(\cdot-z)\|_{H^1}+\|w_c\|_{H^1}\rightarrow 0
\]
as $c\rightarrow \infty$.
\end{proof}

\begin{prop}\label{annul2}
Let $3<p<6$. For a fixed $c\in(\sqrt{\frac{\mu}{m}},\infty)$, suppose that for some $b\in \R$, there exists a sequence $\{u_j\}\subset H^1$ satisfying
\begin{align*}
&u_j\in N_d(\mathcal{X}),\\
&\|\tilde{I}_c^\prime(u_j)\|_{H^{-1}}\rightarrow 0,\\
&\tilde{I}_c(u_j)\rightarrow b \mbox{ as } j\rightarrow \infty,
\end{align*} 
where $d>0$ is a constant.
Then for small  $d>0$,  $b$ is a critical value of $\tilde{I}_c$, and the sequence $\{u_j(\cdot+x_j)\}_{j=1}^\infty\subset H^1$ has a strongly convergent subsequence in $H^1$, where $x_j\in \R^3$.
\end{prop}
\begin{proof}
Since $u_j\in N_d(\mathcal{X})$, $\{u_j\}_{j=1}^\infty$ is bounded in $H^1$. Then we can extract a subsequence such that $\tilde{u}_{j_k}\equiv u_{j_k}(\cdot+x_{j_k})$ converges to $u_0\not\equiv0$ weakly in $H^1$ as $k\rightarrow \infty$, where $x_{j_k}\in \R^3$. It is standard to show that $u_0$ is a critical point of $I_c$.

Next, we show $\tilde{u}_{j_k}\rightarrow u_0$ in $H^1$ as $k\rightarrow \infty$. By Proposition \ref{arq1}, there exists $R_0>0$ such that  
\begin{equation}\label{1qqp1}
\|\tilde{u}_{j_k}\|_{H^1(\R^3\setminus B(0,R_0))}\le 2d.
\end{equation}
We choose a function $\zeta\in C^\infty(\R^3)$ such that
\[
\zeta(x)=\begin{cases} 1 &\mbox{ for } |x|\ge 2R_0, \\ 0 &\mbox{ for } |x|\le R_0.\end{cases}
\]
 Since $\tilde{I}_c^\prime(\tilde{u}_{j_k})(\zeta(\tilde{u}_{j_k}-u_0))-\tilde{I}_c^\prime(u_0)(\zeta(\tilde{u}_{j_k}-u_0))\rightarrow 0$ as $k\rightarrow \infty$, we deduce that
\begin{equation}\label{1qqp2}
\begin{aligned}
 &\int_{\R^3\setminus B(0,2R)}|\nabla (\tilde{u}_{j_k}-u_0)|^2+\Big(2m\mu-\frac{\mu^2}{c^2}\Big)(\tilde{u}_{j_k}-u_0)^2dx\\
&\le \int_{\R^3\setminus B(0,2R)} \Big(\frac{q}{c}\Big)^2(\tilde{u}_{j_k}- u_0) (\tilde{u}_{j_k} \Phi_{\tilde{u}_{j_k}}^2-u_0\Phi_{u_0}^2)
+2q\Big(m-\frac{\mu}{c^2}\Big)(\tilde{u}_{j_k}-u_0)(\tilde{u}_{j_k}\Phi_{\tilde{u}_{j_k}}-u_0\Phi_{u_0})\\
&\qquad +(\tilde{u}_{j_k}-u_0)((\tilde{u}_{j_k})_+^{p-1} -(u_0)_+^{p-1}  )dx+o(1)
\end{aligned}
\end{equation}
as $k\rightarrow \infty$.
We note that, by Lemma \ref{aq1},
\begin{equation}\label{1qqp3}
\begin{aligned}
&\int_{\R^3\setminus B(0,2R)}(v- w) (v \Phi_{v}^2-w\Phi_{w}^2)dx\\
&\le \big(\|\Phi_v\|_{L^6}^2\|v\|_{L^3(\R^3\setminus B(0,2R))}+\|\Phi_w\|_{L^6}^2\|w\|_{L^3(\R^3\setminus B(0,2R))}\big)\|v-w\|_{L^3(\R^3\setminus B(0,2R))}\\
&\le C_1\big(\|v\|_{H^1}^4\|v\|_{H^1(\R^3\setminus B(0,2R))}+\| w\|_{H^1}^4\|w\|_{H^1(\R^3\setminus B(0,2R))}\big)\|v-w\|_{H^1(\R^3\setminus B(0,2R))},
\end{aligned}
\end{equation}
\begin{equation}\label{1qqp4}
\begin{aligned}
&\int_{\R^3\setminus B(0,2R)}(v- w) (v \Phi_{v}-w\Phi_{w} )dx\\
&\le \big(\|\Phi_v\|_{L^6} \|v\|_{L^3(\R^3\setminus B(0,2R))}+\|\Phi_w\|_{L^6} \|w\|_{L^3(\R^3\setminus B(0,2R))}\big)\|v-w\|_{L^2(\R^3\setminus B(0,2R))}\\
&\le C_2\big(\|v\|_{H^1}^2\|v\|_{H^1(\R^3\setminus B(0,2R))}+\| w\|_{H^1}^2\|w\|_{H^1(\R^3\setminus B(0,2R))}\big)\|v-w\|_{H^1(\R^3\setminus B(0,2R))},
\end{aligned}
\end{equation}
and
\begin{equation}\label{1qqp5}
\begin{aligned}
&\int_{\R^3\setminus B(0,2R)}((v)_+^{p-1} -(w)_+^{p-1})(v-w)dx=(p-1)\int_{\R^3\setminus B(0,2R)}(tv+(1-t)w)_+^{p-2}(v-w)^2dx\\
&\le (p-1)\|tv+(1-t)w\|_{L^p(\R^3\setminus B(0,2R))}^{p-2}\|v-w\|_{L^p(\R^3\setminus B(0,2R))}^2\\
&\le C_3\big(\| v\|_{H^1(\R^3\setminus B(0,2R))}^{p-2}+\|w\|_{H^1(\R^3\setminus B(0,2R))}^{p-2}\big)\|v-w\|_{H^1(\R^3\setminus B(0,2R))}^2,
\end{aligned}
\end{equation}
where $t\in [0,1]$. Then, by \eqref{1qqp1}-\eqref{1qqp5}, we see that for small $d>0$, 
\begin{equation}\label{rrr}
\|\tilde{u}_{j_k}-u_0\|_{H^1(\R^3\setminus B(0,2R))}\rightarrow 0
\end{equation} as $k\rightarrow \infty$. Thus, by \eqref{rrr} and the Rellich-Kondrachov compactness theorem, we see that $\tilde{u}_{j_k}\rightarrow u_0$ in $H^1$ as $k\rightarrow \infty$.
\end{proof}
\begin{prop}\label{prop3}
For $3< p<6$, there exist $\bar{c}_0>0$ and $\bar{d}_0>0$ such that for $c>\bar{c}_0$ and for $0<d<\bar{d}_0$, $\tilde{I}_c$ has a critical point $u$ in $N_d(\mathcal{X})$ with $\tilde{I}_c(u)\le \hat{e}_c$.
\end{prop}
\begin{proof}
Arguing indirectly, suppose $\tilde{I}_c^\prime(u)\neq 0$ for $u\in N_d(\mathcal{X})$ with $\tilde{I}_c(u)\le \hat{e}_c$.
By Proposition \ref{annul1} and  Proposition \ref{annul2}, we can  take positive constants $\bar{c}_0$ and $\bar{d}_0$ such that for $c>\bar{c}_0$ and for $0<d<\bar{d}_0$,
\[
\|\tilde{I}_c^\prime(u)\|_{H^{-1}}\ge \nu
\] 
for $u\in N_d(\mathcal{X})\setminus N_{d/2}(\mathcal{X})$ with $\tilde{I}_c(u)\le \hat{e}_c$, and
\[
\|\tilde{I}_c^\prime(u)\|_{H^{-1}}\ge \sigma_c
\] 
for $u\in N_d(\mathcal{X})$ with $\tilde{I}_c(u)\le \hat{e}_c$, where $\nu>0$ is a constant independent of $c$, and $\sigma_c>0$ is a constant depending on $c$. 
Then, by a deformation argument using Proposition \ref{prop-upper-estimate} and Proposition \ref{path1} (see Proposition 7 in \cite{BJ} for a detailed argument), we get a contradiction. 

\end{proof}

\section{Nonrelativistic limit of ground states for $3 < p < 6$}\label{limm1}
In this section, we complete the proof of Theorem \ref{mthm1}.
By Proposition \ref{gmmq1}, Proposition \ref{prop-upper-estimate} and Proposition \ref{prop3}, we see that for every $3< p < 6$, there exists a ground state solutions $u_c$ to \eqref{sta1} such that
\begin{equation}\label{energy bound}
\ \limsup_{c\to\infty}I_c(u_c)  \leq E_\infty.
\end{equation}

\begin{prop}\label{b1q}
Let $3< p<6$ and  $u_c$ be a  ground state solution of \eqref{sta1}. Then we have
\begin{align*}
\sup_{c>\sqrt{\frac{ \mu}{ m}}}\|u_c\|_{H^1}\le C \mbox{ and } \inf_{c>\sqrt{\frac{ \mu}{ m}}}\|u_c\|_{L^p}\ge \frac{1}{C},
\end{align*}
where $C>0$ is a constant independent of $c$.
\end{prop}
\begin{proof}
 We note by \eqref{energy bound} that  
\begin{equation}\label{vvp1}
\begin{aligned}
C_1&\ge \frac{5p-12}{2}I_c(u_c)-J_c(u_c)+\frac{4-p}{2}P_c(u_c)\\
&=\int_{\R^3}(p-3)|\nabla u_c|^2+\frac{p-2}{2}\Big(2m\mu-\frac{\mu^2}{c^2}\Big)u_c^2+\frac{p-2}{2}\Big(\frac{q}{c}\Big)^2 u_c^2\Phi_{u_c}^2dx,
\end{aligned}
\end{equation}where $C_1>0$ is a constant independent of $c$.
This implies $\|u_c\|_{H^1}$ is bounded uniformly in $c>\sqrt{\frac{ \mu}{ m}}$. Moreover,  since $J_c(u_c)=0$  and $ -\frac{1}{q}(c^2m-\mu)\le \Phi_{u_c}\le 0$, we have for $c>\sqrt{\frac{ \mu}{ m}}$,
\begin{equation}\label{m1}
\begin{aligned}
\int_{\R^3}|u_c|^p&=\int_{\R^3}|\nabla u_c|^2+\Big(2m\mu-\frac{\mu^2}{c^2}\Big)u_c^2-\Big(\frac{q}{c}\Big)^2\Phi_{u_c} u_c^2\bigg( \Phi_{u_c}+\Big(\frac{c}{q}\Big)^22q\Big(m-\frac{\mu}{c^2}\Big) \bigg)dx\\
&\ge \int_{\R^3}|\nabla u_c|^2+m\mu u_c^2dx+\Big(\frac{q}{c}\Big)^2|\Phi_{u_c}| u_c^2\Big( \Phi_{u_c}+2 \frac{1}{q}  (c^2m- \mu ) \Big)dx\\
&\ge \int_{\R^3}|\nabla u_c|^2+m\mu u_c^2dx \ge  {C_2}\Big(\int_{\R^3}|u_c|^pdx\Big)^{2/p},
\end{aligned}
\end{equation}
 where $C_2$ is a positive constant indendent of $c$.
Then we have $\int_{\R^3}|u_c|^pdx \ge \frac{1}{C}$, where $C$ is a  positive constant indendent of $c$.
\end{proof}

\begin{prop}\label{lema2}
For $3< p <6$, let $\{u_c\}_{c>\sqrt{\frac{ \mu}{ m}}}\subset H^1$ be a ground state solution of \eqref{sta1}. 
Then there exists a sequence $\{x_c\} \in \R^3$ such that 
$\bar{u}_c(\cdot)\equiv u_c(\cdot+x_c)$ converges to $u_\infty$   in $H^1(\R^3)$ as $c\rightarrow \infty$, up to a subsequence, 
where $u_\infty$ is a ground state solution of \eqref{a1}.
\end{prop}
\begin{proof}
By Proposition \ref{b1q} and \cite[Lemma 1.1]{L}, we have
\begin{equation*}
\sup_{x\in \R^3}\int_{B_1(x)}|u_c|^2dx=\int_{B_1(x_c)}|u_c|^2dx\ge \bar{C}>0,
\end{equation*}
where $\bar{C}$ is a constant indepnedent of $c$.

 It follows from Proposition \ref{b1q} that $\{u_c\}_{c>\sqrt{\frac{ \mu}{ m}}}$ is bounded in $H^1$ uniformly in $c$. Then we may assume $\bar{u}_c\equiv u_c(\cdot+x_c)$ converges to $u_\infty\not\equiv0$ weakly in $H^1$ and strongly in $L_{loc}^q(\R^3)$, where $0<q<6$. Let $\Phi_{\bar{u}_c}$ be the solution of
\[
-\Delta \Phi+\frac{q^2}{c^2}\bar{u}_c^2\Phi=-q(m-\frac{\mu }{c^2}) \bar{u}_c^2 \mbox{ in } \R^3.
\]
Since $\|\Phi_{\bar{u}_c}\|_{D^{1,2}}\le C_1q(m-\frac{\mu }{c^2}) \|\bar{u}_c\|_{H^1}^2\le C_2,$ where $C_1, C_2>0$ are constants independent of $c$, we may assume that
\[
\Phi_{\bar{u}_c}\rightharpoonup \phi_{u_\infty} \mbox{ weakly in } D^{1,2} \mbox{ and } \Phi_{\bar{u}_c}\rightarrow \phi_{u_\infty} \mbox{ in } L^q_{loc}(\R^3),
\]
as $c\rightarrow \infty$, where $0<q<6$ and $\phi_{u_\infty}$ is a weak solution of $
-\Delta \phi+qmu_\infty^2=0.$
 Then it is standard to show that $u_\infty$ is a non-trivial weak solution of \eqref{a1}.

Next, we claim that $u_\infty$ is a ground state solution of \eqref{a1}.  We note that, since $u_\infty$ is a non-trivial weak solution of \eqref{a1}, we have 
\[
J_\infty(u_\infty)= P_\infty(u_\infty) =0
\]
and
\begin{equation}\label{vvp3}
\frac{5p-12}{2}I_\infty(u_\infty)-J_\infty(u_\infty)+\frac{4-p}{2}P_\infty(u_\infty)=\int_{\R^3}(p-3)|\nabla u_\infty|^2+ (p-2)m\mu u_\infty^2.
\end{equation}
Then, by \eqref{energy bound}, \eqref{vvp1} and \eqref{vvp3}, we have 
\begin{align*}
\frac{5p-12}{2}E_\infty&\ge \frac{5p-12}{2}\liminf_{c\rightarrow \infty}I_c( {u}_c)\\
&=\liminf_{c\rightarrow \infty}\Big(\int_{\R^3}(p-3)|\nabla u_c|^2+\frac{p-2}{2}\Big(2m\mu-\frac{\mu^2}{c^2}\Big)u_c^2+\frac{p-2}{2}\Big(\frac{q}{c}\Big)^2 u_c^2\Phi_{u_c}^2dx\Big)\\
&\ge \int_{\R^3}(p-3)|\nabla u_\infty|^2+(p-2)m\mu u_\infty^2 =\frac{5p-12}{2} I_\infty(u_\infty),
\end{align*}
which proves the claim.

Finally, to prove the strong convergence in $H^1(\R^3)$, we note that,
by \eqref{energy bound}, \eqref{vvp1}, \eqref{vvp3}, Proposition \ref{b1q}  and the fact that $\bar{u}_c$ converges to $u_\infty\not\equiv0$ weakly in $H^1$,   
\begin{align*}
\frac{5p-12}{2}E_\infty &\ge\frac{5p-12}{2}\lim_{c\rightarrow \infty}I_c(\bar{u}_c)\\
&=\lim_{c\rightarrow \infty}\int_{\R^3}(p-3)|\nabla \bar{u}_c|^2+\frac{p-2}{2}\Big(2m\mu-\frac{\mu^2}{c^2}\Big)\bar{u}_c^2+\frac{p-2}{2}\Big(\frac{q}{c}\Big)^2 \bar{u}_c^2\Phi_{\bar{u}_c}^2dx\\
&=\int_{\R^3}(p-3)|\nabla u_\infty|^2+ (p-2)m\mu u_\infty^2dx\\
&\qquad +\lim_{c\rightarrow \infty}\int_{\R^3}(p-3)|\nabla (\bar{u}_c-u_\infty)|^2+(p-2)m\mu (\bar{u}_c-u_\infty)^2dx\\
&=\frac{5p-12}{2}E_\infty+\lim_{c\rightarrow \infty}\int_{\R^3}(p-3)|\nabla (\bar{u}_c-u_\infty)|^2+(p-2)m\mu (\bar{u}_c-u_\infty)^2dx.
\end{align*}
From this, we deduce that $ \bar{u}_c\rightarrow u_\infty$ in $H^1$ as $c\rightarrow \infty$, up to a subsequence. This completes the proof. 
\end{proof}

\begin{proof}[\bf Proof of Theorem \ref{mthm1}]
It is sufficient to show $H^2$ convergence of $\bar{u}_c$ to $u_\infty$. We may rewrite $\bar{u}_c$ as $u_c$.
%Let  $v_{\infty,q}$ and $v_{c,q}$ be   solutions of \eqref{a1} and \eqref{sta1}, respectively, such that 
%\[
%v_{c,q}\rightarrow v_{\infty,q} \mbox{ in } H^1
%\] as $c\rightarrow \infty$. Then we have 
%\[
%v_{c,q}\rightarrow v_{\infty,q} \mbox{ in } H^2(\R^3)
%\] as $c\rightarrow \infty$.
We note that, by Lemma \ref{aq1} and \cite[Theorem 4.1]{HL}, for  $u\in H^1$,
\[
\sup_{x\in\Omega} |\Phi_u(x)|\le C_1\|u\|_{H^1}^2 \ \ \mbox{ and }\ \ 
 \|  |u|^{p-2} \|_{L^{\frac{6}{p-2}}(\Omega)} =\|u\|_{L^{{6}}(\Omega)}^{p-2}\le C_2\|u\|_{H^1}^{p-2},
\] 
where  $\Omega$ is bounded domain in $\R^3$, and $C_1$ and $C_2$ are   positive constants independent of $u$ and $\Omega$. Then, since $\{\|u_{c}\|_{H^1}\}_{c}$ is bounded, we see that $\{\|u_{c}\|_{L^\infty}\}_{c}$ is bounded (see \cite[Theorem 4.1]{HL}).

Since $u_{\infty}$ and $u_{c}$ are   solutions   of \eqref{a1} and \eqref{sta1} respectively, we have
\begin{equation}\label{imi}
\begin{aligned}
-\Delta(u_{c}-u_{\infty})&=  -2m\mu(u_{c}-u_{\infty})+\Big(\frac{\mu}{c}\Big)^2u_{c}+\Big(\frac{q}{c}\Big)^2u_{c} \Phi_{u_{c}}^2-2q\frac{\mu}{c^2} u_{c}\Phi_{u_{c}}\\
&\qquad +2qm(u_{c}\Phi_{u_{c}}-u_{\infty}\phi_{u_{\infty}})+|u_{c}|^{p-2}u_{c}-|u_{\infty}|^{p-2}u_{\infty}.
\end{aligned}
\end{equation}
We note that,  by Lemma \ref{iiop}, Lemma \ref{b2bb}, Lemma \ref{yyy1} and Proposition \ref{lema2},  
\begin{equation}\label{imi1}
\begin{aligned}
&\|u_{c}\Phi_{u_{c}}-u_{\infty}\phi_{u_{\infty}}\|_{L^2}\\
&=\|u_{c}(\Phi_{u_{c}}-\phi_{u_{c}})+(u_{c}-u_{\infty})\phi_{u_{c}}+(\phi_{u_{c}}-\phi_{u_{\infty}})u_{\infty}\|_{L^2}\\
&\le \|u_{c}\|_{L^3}\|\Phi_{u_{c}}-\phi_{u_{c}}\|_{L^6}+\| u_{c}-u_{\infty}\|_{L^3}\|\phi_{u_{c}}\|_{L^6}+\| \phi_{u_{c}}-\phi_{u_{\infty}}\|_{L^6}\|u_{\infty}\|_{L^3}\rightarrow 0 
\end{aligned}
\end{equation} 
as $c\rightarrow \infty$, and by the fact that $\{\|u_{c}\|_{L^\infty}\}_{c}$ is bounded,
\begin{equation}\label{imi2}
\big\||u_{c}|^{p-2}u_{c}-|u_{\infty}|^{p-2}u_{\infty}\big\|_{L^2}=(p-1)\big\||u_{\infty}+t(u_{c}-u_{\infty})|^{p-2}
(u_{c}- u_{\infty})\big\|_{L^2}\rightarrow 0
\end{equation}
as $c\rightarrow \infty$, where $t\in [0,1]$.
Thus, by \eqref{imi}-\eqref{imi2} and the Calder\'{o}n–Zygmund inequality, we have
\begin{align*}
\|u_{c}-u_{\infty}\|_{H^2(\R^3)}=\|-\Delta(u_{c}-u_{\infty})\|_{L^2}+o(1)=o(1)
\end{align*}as $c\rightarrow \infty$.
\end{proof}

\section{Nonrelativistic limit of two positive solutions for $2<p< 3$}\label{case23}
In this section, we will construct two radially symmetric positive solutions of NMKG for $2<p< 3$. We prove first  the existence of a radially symmetric positive solution $v_{c,q}$ of \eqref{sta1} satisfying
\[
 \lim_{c\to\infty}\|v_{c,q}-v_\infty\|_{H^1}= 0,
\]
where $v_\infty$ is a global minimizer of $I_\infty$.

We assume $2<p< 3$ and denote 
\[
e_\infty\equiv \inf_{u\in H_r^1}\tilde{I}_\infty(u),\ \ \mathcal{X} _r\equiv \{u\in H_r^1\  | \  \tilde{I}_\infty(u)=e_\infty \} 
\] 
and 
\[
N_d(\mathcal{X}_r)\equiv \{u\in H_r^1\ | \ \inf_{v\in \mathcal{X}_r}\|u-v\|_{H^1}\le d\},
\]
where   $d>0$ is a constant.  We remark that, by   \cite[Theorem 4.3, Corollary 4.4]{R}, $\mathcal{X}_r$ is bounded in $H^1$, and for small $q>0$, $e_\infty<0$ and $\mathcal{X}_r\neq\emptyset$. Moreover, since $e_\infty<0$ for small $q>0$, and  for $u\in \mathcal{X}_r$,
\begin{align*}
e_\infty=\tilde{I}_\infty(u)&=\frac12\int_{\R^3}|\nabla u|^2+2m\mu u^2-qmu^2\phi_u dx-\frac{1}{p}\int_{\R^3}(u)_+^pdx\\
&\ge \frac12\int_{\R^3}|\nabla u|^2+2m\mu u^2  dx-\frac{C_1}{p}\Big(\int_{\R^3}|\nabla u| ^2+u^2dx\Big)^{p/2},
\end{align*}
where $C_1>0$ is a constant independent of $u\in \mathcal{X}_r$, we see that there exists $\hat{q}_0>0$ such that for $0<q<\hat{q}_0$, $\mathcal{X}_r\neq\emptyset$ and
\begin{equation}\label{zem1}
\inf_{u\in \mathcal{X}_r}\|u\|_{H^1}>\hat{d}_0>0,
\end{equation}
where $\hat{d}_0$ is a positive constant. Taking $d\in (0,\frac{\hat{d}_0}{2}),$ we deduce that for $0<q<\hat{q}_0$,  $0\notin N_d(\mathcal{X}_r)$. For $d\in (0,\frac{\hat{d}_0}{2})$ and   $0<q<\hat{q}_0$, take $V_0\in \mathcal{X}_r$ and set 
$$\alpha_c=\inf_{u\in N_d(\mathcal{X}_r)}\tilde{I}_c(u) \ \mbox{ and }\ m_c=\tilde{I}_c(V_0).$$
 Clearly, we have $m_c\ge \alpha_c$.  
We try to find a critical point of $\tilde{I}_c$ in $N_d(\mathcal{X}_r)$.
\begin{prop}\label{llla1}
For $2<p< 3$, $0<q<\hat{q}_0$ and   $d\in (0,\frac{\hat{d}_0}{2})$,  we have
\[
\liminf_{c\rightarrow \infty}\alpha_c\ge e_\infty.
\]
\end{prop}
\begin{proof}
It is standard to show that there exists $v_c\in N_d(\mathcal{X}_r)$ such that
\[
\alpha_c=\tilde{I}_c(v_c),
\] because  $\mathcal{X}_r$ is bounded in $H^1$.   Since $v_c$ is bounded in $H^1_r$ uniformly in $c$, we assume that $v_c$ converges to $v$ in $L^s$ and weakly in $H^1$ as $c\rightarrow \infty$, where $s\in (2,6)$ and $v\in N_d(\mathcal{X}_r)$. Then, by Lemma \ref{yyy1}, we have
\begin{align*}
\liminf_{c\rightarrow \infty}\alpha_c&=\liminf_{c\rightarrow \infty}\tilde{I}_c(v_c)\\
&= \liminf_{c\rightarrow \infty}\Big[\frac12\int_{\R^3}|\nabla v_c|^2+\Big(2m\mu-\frac{\mu^2}{c^2}\Big)v_c^2-q\Big(m-\frac{\mu}{c^2}\Big) v_c^2 \Phi_{v_c} dx-\frac{1}{p}\int_{\R^3}(v_c)_+^pdx\Big]\\
&\ge\frac12\int_{\R^3}|\nabla v|^2+2m\mu v^2-qmv^2\phi_v dx-\frac{1}{p}\int_{\R^3}(v)_+^pdx =\tilde{I}_\infty(v)\ge e_\infty.
\end{align*}
\end{proof}

\begin{prop}\label{uuu1}
For $2<p< 3$ and $0<q<\hat{q}_0$,  we have
\[
m_c\rightarrow e_\infty
\]uniformly in $q$ as $c\rightarrow \infty$.
\end{prop}
\begin{proof}
By Lemma \ref{yyy1},
\begin{align*}
\tilde{I}_c(V_0)&=\frac12\int_{\R^3}|\nabla V_0|^2+\Big(2m\mu-\frac{\mu^2}{c^2}\Big)V_0^2-q\Big(m-\frac{\mu}{c^2}\Big) V_0^2 \Phi_{V_0} dx-\frac{1}{p}\int_{\R^3}(V_0)_+^pdx\\
&=\frac12\int_{\R^3}|\nabla V_0|^2+2m\mu V_0^2-qmV_0^2\phi_{V_0} dx-\frac{1}{p}\int_{\R^3}(V_0)_+^pdx+o(1)\\
&=\tilde{I}_\infty(V_0)+o(1)=e_\infty+o(1)
\end{align*}
as $c\rightarrow \infty$.
\end{proof}

\begin{prop}\label{rannul1}
Let $2<p< 3$, $0<q<\hat{q}_0$ and   $d\in (0,\frac{\hat{d}_0}{2})$. For large $c>0$    and for any $d^\prime\in (0,d)$, there exists $\nu_0\equiv \nu_0(d,d^\prime)>0$ independent of $c>0$ such that 
\[
\inf\{\|\tilde{I}_c^\prime(u)\|_{H^{-1}}\ |\ \tilde{I}_c(u)\le m_c, u\in N_d(\mathcal{X}_r)\setminus N_{d^\prime}(\mathcal{X}_r)\}\ge \nu_0>0.
\]
\end{prop}
\begin{proof}
Let $\{c_i\}_{i=1}^\infty$ be such that $\lim_{i\rightarrow \infty}c_i=\infty$. It suffices to show that  if
\begin{align*}
u_{c_i}\in N_{d}(\mathcal{X}_r),\ \ \tilde{I}_{c_i}(u_{c_i})\le m_{c_i},\ \mbox{ and }\ \|\tilde{I}^\prime_{c_i}(u_{c_i})\|_{H^{-1}}\rightarrow 0 
\end{align*}
as $i\rightarrow \infty$, then 
\[
\inf_{v\in \mathcal{X}_r}\|u_{c_i}-v\|_{H^1}\rightarrow 0 \mbox{ as } i\rightarrow \infty.
\]
 For the sake of simplicity of notation, we write $c$ for $c_i$. Since $\{u_c\}\subset H_r^1$ is bounded in $H^1$, we see that $u_c$ converges to $u$ in $L^s$ and weakly in $H^1$ as $c\rightarrow \infty$, up to a subsequence, where $s\in(2,6)$. Then, by Lemma \ref{yyy1} and Proposition \ref{uuu1}, we have
\begin{align*}
e_\infty&=\liminf_{c\rightarrow \infty} m_c \ge \liminf_{c\rightarrow \infty}\tilde{I}_c(u_c)\\
&= \liminf_{c\rightarrow \infty}\Big[\frac12\int_{\R^3}|\nabla u_c|^2+\Big(2m\mu-\frac{\mu^2}{c^2}\Big)u_c^2-q\Big(m-\frac{\mu}{c^2}\Big) u_c^2 \Phi_{u_c} dx-\frac{1}{p}\int_{\R^3}(u_c)_+^pdx\Big]\\
&\ge \frac12\int_{\R^3}|\nabla u|^2+2m\mu u^2-qmu^2\phi_u dx-\frac{1}{p}\int_{\R^3}(u)_+^pdx =\tilde{I}_\infty(u),
\end{align*}
which implies that $e_\infty=\tilde{I}_\infty(u)$. 
%By the principle of symmetric criticality (see \cite{P}), we see that $\tilde{I}_\infty^\prime(u)=0$.

We claim that $u_c\rightarrow u$ in $H^1$. Indeed, by Lemma \ref{yyy1} and the fact that $\|\tilde{I}^\prime_c(u_c)\|_{H^{-1}}\rightarrow 0$ as $c\rightarrow \infty$, we see that
\begin{equation}\label{uup1a}
\begin{aligned}
o(1)&=\tilde{I}_c^\prime (u_c)u\\
&=\int_{\R^3}\nabla u_c\cdot \nabla u+\Big(2m\mu-\frac{\mu^2}{c^2}\Big)u_c u-\Big(\frac{q}{c}\Big)^2 u_cu \Phi_{u_c}^2-2q\Big(m-\frac{\mu}{c^2}\Big)   u_c  u\Phi_{u_c}-(u_c)_+^{p-1}  udx\\
&=\int_{\R^3} |\nabla u|^2+2m\mu u^2-2qm u^2\phi_u-(u)_+^{p} dx+o(1)
\end{aligned}
\end{equation}
as $c\rightarrow \infty$, and 
\begin{equation}\label{uup2a}
\begin{aligned}
o(1)&=\tilde{I}_c^\prime (u_c)u_c\\
&=\int_{\R^3}|\nabla u_c|^2+\Big(2m\mu-\frac{\mu^2}{c^2}\Big)u_c^2-\Big(\frac{q}{c}\Big)^2 u_c^2 \Phi_{u_c}^2-2q\Big(m-\frac{\mu}{c^2}\Big) u_c^2\Phi_{u_c}-(u_c)_+^{p} dx\\
&=\int_{\R^3}|\nabla u_c|^2+2m\mu u_c^2+2qmu^2\phi_u-(u)_+^pdx+o(1)
\end{aligned}
\end{equation}
as $c\rightarrow \infty$. Thus, by \eqref{uup1a} and \eqref{uup2a}, we have $u_c\rightarrow u$ in $H^1$.

\end{proof}

\begin{prop}\label{rannul2} 
Let $2<p< 3$, $0<q<\hat{q}_0$ and $d\in (0,\frac{\hat{d}_0}{2}).$ For a fixed $c\in(\sqrt{\frac{\mu}{m}},\infty)$, suppose that for some $b\in \R$, there exists a sequence $\{u_j\}\subset H_r^1$ satisfying
\begin{align*}
&u_j\in N_d(\mathcal{X}_r),\\
&\|\tilde{I}_c^\prime(u_j)\|_{H^{-1}}\rightarrow 0,\\
&\tilde{I}_c(u_j)\rightarrow b \mbox{ as } j\rightarrow \infty.
\end{align*} 
Then    $b$ is a critical value of $\tilde{I}_c$, and the sequence $\{u_j\}_{j=1}^\infty\subset H_r^1$ has a strongly convergent subsequence in $H^1$.
\end{prop}
\begin{proof}
Since $\{u_j\}\subset N_d(\mathcal{X}_r)$ is bounded in $H^1$, we see that $u_j$ converges to $u$ in $L^s$ and weakly in $H^1$ as $c\rightarrow \infty$, up to a subsequence, where $s\in(2,6)$. It is standard to show that $u$ is a critical point of $\tilde{I}_c$.

We claim that $u_j\rightarrow u$ in $H^1$. Indeed, by Lemma \ref{aaq1} and the fact that $\|\tilde{I}_c^\prime(u_j)\|_{H^{-1}}\rightarrow 0$ as $j\rightarrow \infty$, we have
\begin{align*}
o(1)&=\tilde{I}_c^\prime (u_j)u_j\\
&=\int_{\R^3}|\nabla u_j|^2+\Big(2m\mu-\frac{\mu^2}{c^2}\Big)u_j^2-\Big(\frac{q}{c}\Big)^2 u_j^2 \Phi_{u_j}^2-2q\Big(m-\frac{\mu}{c^2}\Big) u_j^2\Phi_{u_j}-|u_j|^{p} dx\\
&=\int_{\R^3}|\nabla u_j|^2+\Big(2m\mu-\frac{\mu^2}{c^2}\Big)u_j^2-\Big(\frac{q}{c}\Big)^2 u^2 \Phi_{u}^2-2q\Big(m-\frac{\mu}{c^2}\Big) u^2\Phi_{u}-|u|^{p} dx+o(1)
\end{align*}
as $j\rightarrow \infty$ and
\begin{align*}
 0=\tilde{I}_c^\prime(u)u=\int_{\R^3}|\nabla u|^2+\Big(2m\mu-\frac{\mu^2}{c^2}\Big)u^2-\Big(\frac{q}{c}\Big)^2 u^2 \Phi_u^2-2q\Big(m-\frac{\mu}{c^2}\Big) u^2\Phi_u-|u|^{p} dx.
\end{align*}
Thus, we deduce  that $u_j\rightarrow u$ in $H^1$ as $j\rightarrow \infty$.
\end{proof}
\begin{prop}\label{apro1}
Let $2<p< 3$, $0<q<\hat{q}_0$ and   $d\in (0,\frac{\hat{d}_0}{2})$. Then there exists $\hat{c}_0>0$  such that for $c>\hat{c}_0$, $\tilde{I}_c$ has a non-trivial critical point $u$ in $N_d(\mathcal{X}_r)$ with $\tilde{I}_c(u)\le m_c$.
\end{prop}
\begin{proof}
Assume that $2<p< 3$, $0<q<\hat{q}_0$ and   $d\in (0,\frac{\hat{d}_0}{2})$. Suppose $\tilde{I}_c^\prime(u)\neq 0$ for $u\in N_d(\mathcal{X}_r)$ with $\tilde{I}_c(u)\le m_c$.
By Proposition \ref{llla1}--\ref{rannul2}, we can  take a positive constant  $\hat{c}_0$  such that for $c>\hat{c}_0$ and for $0<q<\hat{q}_0$,
\begin{equation}\label{zzzi1}
\alpha_c\ge e_\infty-\epsilon_1,\ \  |m_c-e_\infty|\le \epsilon_1,
\end{equation}
\begin{equation}\label{zzzi2}
\|\tilde{I}_c^\prime(u)\|_{H^{-1}}\ge \nu_0
\end{equation}
for $u\in N_{\frac23 d}(\mathcal{X}_r)\setminus N_{\frac13 d}(\mathcal{X}_r)$ with $\tilde{I}_c(u)\le m_c$, and
\begin{equation}\label{zzzi3}
\|\tilde{I}_c^\prime(u)\|_{H^{-1}}\ge \hat{\sigma}_c
 \end{equation}for $u\in N_d(\mathcal{X}_r)$ with $\tilde{I}_c(u)\le m_c$, where $d\in (0,\frac{\hat{d}_0}{2}),$ $\epsilon_1\in(0,\frac{d\nu_0}{6})$, and $\hat{\sigma}_c>0$ is a constant depending on $c$. For $u\in N_d(\mathcal{X}_r)$ with $\tilde{I}_c(u)\le m_c$, we consider the following ODE:
\[
\begin{cases}
&\frac{d\eta}{d\tau}=-\varphi_1(\tilde{I}_c(\eta))\varphi_2(dist_{H^1}(\eta, \mathcal{X}_r ))\frac{\tilde{I}_c^\prime(\eta)}{\|\tilde{I}_c^\prime(\eta)\|_{H^{-1}}},\\
&\eta(0,u)=u,
\end{cases}
\]
where 
\[
dist_{H^1}(w, \mathcal{X}_r )=\inf\{\|w-v\|_{H^1}\ |\ v\in \mathcal{X}_r\}
\]
for $w\in H^1$, and $\varphi_1, \varphi_2:\R\rightarrow [0,1]$ are Lipschitz continuous functions such that
\begin{align*}
  &\varphi_1(\xi)=\begin{cases}
                   1 & \mbox{if } \xi\ge e_\infty- \epsilon_1, \\
                   0 & \mbox{if } \xi\le  e_\infty- 2\epsilon_1,
                 \end{cases} 
\ \ \ \  \ \ \varphi_2(\xi)=\begin{cases}
                   1 & \mbox{if } \xi\le \frac23 d, \\
                   0 & \mbox{if } \xi\ge d.
                 \end{cases}
\end{align*}
 Let $T= 3\epsilon_1/\hat{\sigma}_c$ and $V_0\in \mathcal{X}_r$. Since $\tilde{I}_c(\eta(\tau, V_0))\ge \alpha_c\ge e_\infty-\epsilon_1$ for $\tau\in[0,T]$, we deduce that there exists $t_0\in [0,T]$ such that 
\begin{equation}\label{yyq}
dist_{H^1}(\eta(t_0,V_0))=\frac23 d.
\end{equation}
Indeed, if $dist_{H^1}(\eta(\tau,V_0))<\frac23 d$ for $\tau\in [0,T]$, by \eqref{zzzi1} and \eqref{zzzi3},
\begin{align*}
\tilde{I}_c(\eta (T,V_0))&=\tilde{I}_c(V_0)+\int_0^T \frac{d}{d\tau}\tilde{I}_c(\eta(\tau, V_0))d\tau \le e_\infty+\epsilon_1- T\hat{\sigma}_c=e_\infty-2\epsilon_1,
\end{align*}
which is a contradiction. Assume that $t_0$ is the first time that satisfies \eqref{yyq}. Since $\|\frac{d}{d\tau}\eta\|_{H^1}\le 1$, we see that $t_0\ge \frac23 d$ and 
\[
\eta(\tau, V_0)\in N_{\frac23 d}(\mathcal{X}_r)\setminus N_{\frac13 d}(\mathcal{X}_r) \mbox{ for } \tau \in [t_0-\frac13 d, t_0].
\]
Then, by \eqref{zzzi1} and \eqref{zzzi2}, we have
\begin{align*}
\tilde{I}_c(\eta(T,V_0))&=\tilde{I}_c(V_0)+\int_0^T \frac{d}{d\tau}\tilde{I}_c(\eta(\tau, V_0))d\tau \le e_\infty +\epsilon_1+\int_{t_0-\frac13 d}^{t_0} \frac{d}{d\tau}\tilde{I}_c(\eta(\tau, V_0))d\tau\\
&=e_\infty +\epsilon_1-\frac13 d \nu_0<e_\infty-\epsilon_1,
\end{align*}
which is a contradiction.
\end{proof}
\begin{proof}[\bf Proof of Theorem \ref{mthm3}] Let $2<p<3$. By Proposition \ref{apro1} and the proof of Proposition \ref{rannul1}, we   prove the existence of a radially symmetric positive solution $v_{c,q}$ of \eqref{sta1} satisfying 
$$\limsup_{c\rightarrow \infty} \tilde{I}_c(v_{c,q})\le \inf_{u\in H^1_r}\tilde{I}_\infty(u).$$ 
By repeating the same procedure in the proof of Proposition \ref{rannul1}, we can prove  Theorem \ref{mthm3} (ii). 
%Since $\inf_{u\in H^1_r}\tilde{I}_\infty(u)\rightarrow -\infty$ as $q\rightarrow 0$, 

On the other hand, it is known that the ground state solution $w_0$ of the equation
\begin{equation}\label{NLS}
-\Delta u+2m\mu-|u|^{p-2}u=0 \mbox{ in }\R^3
\end{equation} 
is positive, radially symmetric, up to a translation. It is also non-degenerate in the radial class, i.e.,
$\text{Ker} L_0 = \{0\}$, where $L_0:H_r^1\rightarrow H^{-1}_r$ is the linearized operator of \eqref{NLS} at $w_0$, given by $L_0(w)\equiv -\Delta w+2m\mu w-(p-1)|u_0|^{p-2}w.$

Exploiting the non-degeneracy of $w_0$, we see from the implicit function theorem that
there exists of a family of radially symmetric solutions $w_{\infty,q}$ of \eqref{a1} for small $q > 0$ such that $w_{\infty,q}\to w_0$ as $q\rightarrow 0$ in $H^1$. (We refer to \cite{R} for detail.)
As a consequence, one can easily see that $w_{\infty, q}$ is also non-degenerate in the radial class for any small fixed $q > 0$. 
Then one can once more invoke the implicit function theorem to find a family of nontrivial radial solutions $w_{c,q}$ of \eqref{sta1} for large value $c > 0$ and small $q > 0$,
which converges in $H^1$ to $w_{\infty,q}$ as $c\to\infty$. 
This proves  Theorem \ref{mthm3} (i).   
%Following the arguments in Lemma \ref{vno1} (use $w_{\infty,q}$ and $u_0$ instead of $u_{c,q}$ and $u_{\infty,q}$, respectively), we can prove that $w_{\infty,q}$ satisfies \eqref{bm2}. Thus, by Proposition \ref{prop2}, we can prove Theorem \ref{mthm3} (ii).
\end{proof}

\appendix 
\section{Basic estimates}
Here, we provide with several basic estimates, which are repeatedly invoked in the proofs of main theorems.

\begin{lem}\label{iiop}
Let $u\in H^1$. Then we have
\[
\|\phi_u\|_{D^{1,2}}\le Cqm \|u\|_{H^1}^2,
\]
where $C$ is a positive constant.
\end{lem}
\begin{proof}
Let $u\in H^1$. Since $\phi_u$ satisfies
\[
-\Delta\phi_u=-qm u^2 \mbox{ in }\R^3,
\]
we have 
\[
\int_{\R^3}|\nabla \phi_u|^2dx=-qm\int_{\R^3} u^2 \phi_udx\le qm \|\phi_u\|_{L^6}\|u^2\|_{L^{6/5}}\le Cqm\|\phi_u\|_{D^{1,2}}\|u\|_{H^1}^2,
\]
where $C$ is a positive constant.
This implies the result.
\end{proof}

\begin{lem}\label{aq1}
Let $u\in H^1$.  For $c>\sqrt{\frac{\mu}{m}}$, we have
\[
\|\Phi_{u}\|_{D^{1,2}}\le C q\Big(m-\frac{\mu }{c^2}\Big)  \|u\|_{H^1}^2,
\]
where $C$ is a positive constant.
\end{lem}
\begin{proof}
Let $u \in H^1$. Since $\Phi_{u}$ satisfies 
\[
-\Delta \Phi_{u}+ \Big(\frac{q}{c} \Big)^2u^2\Phi_{u}=-q\Big(m-\frac{\mu}{c^2}\Big)   u^2 \mbox{ in } \R^3,
\]
 and 
\begin{equation}\label{vv1}
\|u^2 \Phi_{u}\|_{L^1}\le\|\Phi_{u}\|_{L^6}\|u^2\|_{L^{6/5}} = \|\Phi_{u}\|_{L^6}\|u\|_{L^{12/5}}^2\le C \|\Phi_{u}\|_{D^{1,2}}\|u\|_{H^1}^2,
\end{equation} we have for   $c>\sqrt{\frac{\mu}{m}}$,
\begin{align*}
\|\Phi_{u}\|_{D^{1,2}}^2=\int_{\R^3}|\nabla \Phi_{u}|^2dx&\le -q\Big(m-\frac{\mu}{c^2}\Big)   \int_{\R^3}u^2 \Phi_{u}\\
&\le C q\Big(m-\frac{\mu }{c^2}\Big)  \|u\|_{H^1}^2 \|\Phi_{u}\|_{D^{1,2}},
\end{align*}
where $C$ is a positive constant.
This implies the result.
\end{proof}

\begin{lem}\label{b2bb}
Let $v, w\in  H^1$. Then we have
\[
\|\phi_v-\phi_w\|_{D^{1,2}}\le C\|v+w\|_{H^1}\|v-w\|_{H^1},
\]
where $C=C(q,m)$ is a positive constant.
\end{lem}
\begin{proof}
We note that for $v, w\in  H^1$,
\[
-\Delta(\phi_v-\phi_w)=-qm(v-w)(v+w) \ \mbox{ in }\ \R^3.
\]
Then we have
\[
\|\phi_v-\phi_w\|_{D^{1,2}}\le C\|v+w\|_{H^1}\|v-w\|_{H^1},
\]
where $C=C(q,m)$ is a positive constant.
\end{proof}

\begin{lem}\label{aaq1}
Let $v, w\in H^1$.Then for $c>\sqrt{\frac{\mu}{m}}$, we have 
\[
\|\Phi_v-\Phi_w\|_{D^{1,2}}\le  {C} (\|v\|_{H^1}^2 +1) \|v+w\|_{H^1}\|v-w\|_{L^3},
\]
where $C=C(q,m,\mu)$ is a positive constant.
\end{lem}
\begin{proof}
Since $\Phi_u$ satisfies
\[
-\Delta \Phi_u+\frac{q^2}{c^2}u^2\Phi_u=- q\Big(m-\frac{\mu }{c^2}\Big) u^2 \mbox{ in } \R^3, 
\]
we have 
\[
-\Delta (\Phi_v-\Phi_w)+\frac{q^2}{c^2}w^2(\Phi_v-\Phi_w)=-\frac{q^2}{c^2}(v^2-w^2)\Phi_v- q\Big(m-\frac{\mu }{c^2}\Big)( v^2-w^2) \mbox{ in } \R^3. 
\]
Multiplying $(\Phi_v-\Phi_w)$ to the above equation and then integrating over $\R^3$, we have
\begin{align*}
\int_{\R^3}&|\nabla (\Phi_v-\Phi_w)|^2dx\\
&\le\int_{\R^3}-\frac{q^2}{c^2}(v^2-w^2)\Phi_v(\Phi_v-\Phi_w)- q\Big(m-\frac{\mu }{c^2}\Big)( v^2-w^2)(\Phi_v-\Phi_w)dx\\
&\le \frac{q^2}{c^2}\|v+w\|_{L^3}\|v-w\|_{L^3}\|\Phi_v\|_{L^6}\|\Phi_v-\Phi_w\|_{L^6}\\
&\qquad +q(m-\frac{\mu }{c^2})\|v+w\|_{L^2}\|v-w\|_{L^3}\|\Phi_v-\Phi_w\|_{L^6}\\
&\le {C_1} (\|\Phi_v\|_{D^{1,2}} +1)\|\Phi_v-\Phi_w\|_{D^{1,2}}\|v+w\|_{H^1}\|v-w\|_{L^3}.
\end{align*}
where $C_1=C_1(q,m,\mu)$ is a positive constant. Then, by Lemma \ref{aq1}, for $c>\sqrt{\frac{\mu}{m}}$, 
\[
\|\Phi_v-\Phi_w\|_{D^{1,2}}\le  {C} (\|v\|_{H^1}^2 +1) \|v+w\|_{H^1}\|v-w\|_{L^3},
\]
where $C=C(q,m,\mu)$ is a positive constant.
\end{proof} 

\begin{lem}\label{yyy1}
\[
\|\Phi_v-\phi_w\|_{D^{1,2}}\le C\Big(\frac{1}{c^2}(\|v\|_{H^1}^2+1)\|v\|_{H^1}^2+\|v+w\|_{H^1}\|v-w\|_{L^3}\Big),
\]where $C=C(q,m,\mu)$ is a positive constant.
\end{lem}
\begin{proof}
Since $\phi_w$ and $\Phi_v$   satisfy
\[
-\Delta \phi_w=- qm w^2 \mbox{ in } \R^3 \ \ \mbox{ and  }\ \ -\Delta \Phi_v=-\frac{q^2}{c^2}v^2\Phi_v- q\Big(m-\frac{\mu }{c^2}\Big) v^2 \mbox{ in } \R^3
\]
respectively, we have
\[
-\Delta(\Phi_v-\phi_w)=-\frac{q^2}{c^2}v^2\Phi_v+q\frac{\mu }{c^2}  v^2-qm(v^2-w^2) \mbox{ in } \R^3. 
\]
We multiply $(\Phi_v-\phi_w)$ to the above equation and integrate over $\R^3$ to deduce
\begin{align*}
&\int_{\R^3}|\nabla (\Phi_v-\phi_w)|^2dx\\
&=\frac{1}{c^2}\int_{\R^3}(-q^2v^2\Phi_v+q\mu v^2)(\Phi_v-\phi_w)dx-qm\int_{\R^3}(v^2-w^2)(\Phi_v-\phi_w)dx\\
&\le \frac{1}{c^2}\|\Phi_v-\phi_w\|_{L^6}(q^2\|\Phi_v\|_{L^6}\|v^2\|_{L^{3/2}}+q\mu \|v^2\|_{L^{6/5}}) +qm \|\Phi_v-\phi_w\|_{L^6}\|v+w\|_{L^2}\|v-w\|_{L^3}\\
&\le C_1\|\Phi_v-\phi_w\|_{D^{1,2}}\Big(\frac{1}{c^2} ( \|\Phi_v\|_{D^{1,2}}\|v\|_{H^1}^2+ \|v\|_{H^1}^2) +\|v+w\|_{H^1}\|v-w\|_{L^3}\Big),
\end{align*}
where $C_1=C_1(q,m,\mu)$ is a positive constant. Then, by Lemma \ref{aq1}, we have
\[
\|\Phi_v-\phi_w\|_{D^{1,2}}\le C\Big(\frac{1}{c^2}(\|v\|_{H^1}^2+1)\|v\|_{H^1}^2+\|v+w\|_{H^1}\|v-w\|_{L^3}\Big),
\]where $C=C(q,m,\mu)$ is a positive constant.

\end{proof}

\end{document}